\documentclass[a4paper,11pt]{article}
\usepackage[T1]{fontenc}
\usepackage{lmodern}
\usepackage[utf8]{inputenc}
\usepackage[USenglish]{babel}
\usepackage{amsmath}
\usepackage{amsthm}
\usepackage{mathtools}
\usepackage{amssymb}
\usepackage{commath}
\usepackage[babel=true]{microtype}
\usepackage{mathrsfs}
\usepackage[pdftex]{graphicx}
\usepackage{xcolor}
\usepackage{pgfplots}
\usepackage{tikz}
\usepackage{fancyhdr}
\usepackage{thm-restate}
\usepackage[colorinlistoftodos,color=yellow,textsize=small]{todonotes}
\usepackage{float}
\usepackage[numbers, sort]{natbib}
\usepackage{makecell}
\usepackage{hyperref}
\usepackage[all]{hypcap} 
\usepackage{enumitem}
\usepackage{bm} 
\usepackage[normalem]{ulem}
\usepackage{typearea}
\hypersetup{%
pdftitle=,
pdfauthor=Max Reppen,
pdfcreator=,
pdfproducer=,
colorlinks=true,
linkcolor=black,
citecolor=black,
filecolor=black,
urlcolor=black,
bookmarksopen,
bookmarksopenlevel=1,
}

\pgfplotsset{compat=1.13}
\usetikzlibrary{calc}
\usetikzlibrary{intersections}
\usepgfplotslibrary{fillbetween}
\usetikzlibrary{decorations.text}

\makeatletter
\newtheoremstyle{indented}
  {3pt}
  {3pt}
  {\addtolength{\multlinegap}{2em}\setlength{\leftskip}{2em}\addtolength{\@totalleftmargin}{2em}%
  \setlength{\rightskip}{2em}}
  {}
  {\bfseries}
  {.}
  {.5em}
  {}
\makeatother
\theoremstyle{indented}
\newtheorem{theorem}{Theorem}[section]
\newtheorem{lemma}[theorem]{Lemma}

\newtheorem{corollary}[theorem]{Corollary}

\newtheorem{remark}[theorem]{Remark}
\newtheorem{assumption}[theorem]{Assumption}

\DeclareMathOperator{\Tr}{Tr}

\newcommand{\argmin}{\operatornamewithlimits{arg\,min}}

\newcommand{\R}{\mathbb{R}}
\newcommand{\E}{\mathbb{E}}

\newcommand{\interior}[1]{%
  {\kern0pt#1}^{\mathrm{o}}%
}
\newcommand{\closure}[2][3]{{}\mkern#1mu\overline{\mkern-#1mu#2}}
\newcommand{\bdry}{\partial}
\newcommand{\bigO}{\mathcal{O}}
\newcommand{\smallo}{o}
\newcommand{\1}{1}

\definecolor{lightgray}{gray}{0.66}

\newcommand{\muspace}{\mathcal{M}}
\newcommand{\xmuspace}{\mathcal{O}}
\newcommand{\mudrift}{\kappa}
\newcommand{\musigma}{\tilde{\sigma}}
\newcommand{\muW}{\tilde{W}}
\newcommand{\mumean}{\bar{\mu}}
\newcommand{\xsigma}{\sigma}
\newcommand{\muthresh}{\mu^*}
\newcommand{\process}[1]{(#1_t)_{t \geq 0}}
\newcommand{\rovalue}{V_a}

\title{Optimal dividend policies with random profitability\footnote{The first and third authors were partly supported by the ETH Foundation, the Swiss Finance Institute, and Swiss National Foundation grant SNF 200020\_172815. The second author acknowledges support by the SCOR.}}

\author{
    Max Reppen\thanks{ETH Z\"urich, Department of Mathematics, R\"amistrasse 101, 8092 Z\"urich, Switzerland, email \texttt{max.reppen@math.ethz.ch}.}
    \and
    Jean-Charles Rochet\thanks{SFI@UNIGE and UZH, and TSE/IDEI, email \texttt{jeancharles.rochet@gmail.com}.}
    \and
    H.\ Mete Soner\thanks{ETH Z\"urich, Department of Mathematics, R\"amistrasse 101, 8092 Z\"urich, Switzerland,
    and Swiss Finance Institute, email \texttt{mete.soner@math.ethz.ch}.}
}

\date{March 1, 2018}

\begin{document}

\maketitle

\begin{abstract}
	\noindent We study an optimal dividend problem under a bankruptcy
	constraint. Firms face a trade-off between potential bankruptcy and
	extraction of profits. In contrast to previous works, general cash flow
	drifts, including Ornstein--Uhlenbeck and CIR processes, are considered. We
	provide rigorous proofs of continuity of the value function, whence dynamic
	programming, as well as comparison between the sub- and supersolutions of
	the Hamilton--Jacobi--Bellman equation, and we provide an efficient and
	convergent numerical scheme for finding the solution. The value function is
	given by a nonlinear PDE with a gradient constraint from below in one
	dimension. We find that the optimal strategy is both a barrier and a band
	strategy and that it includes voluntary liquidation in parts of the state
	space. Finally, we present and numerically study extensions of the model,
	including equity issuance and credit lines. 
\end{abstract}

\section{Introduction}
The problem of optimizing dividend flows has its origins in the actuarial field
of ruin theory, which was first treated theoretically by Lundberg
\cite{lundberg1903} in 1903. The theory typically models an insurance firm, and
initially revolved around minimizing the probability of ruin. However, in many
situations in practice, there is an emphasis on maximizing the shareholder
value---an idea which fits well into de Finetti's \cite{deFinetti1957} 1957
proposal to optimize the net present value of dividends paid out until the time
of ruin. With a positive discount rate of the dividends, de Finetti solved the
problem for cash reserves described by a random walk. Since then, this new
class of dividend problems has been extensively studied, especially in the
context of modelling insurance firms.

Although a dividend problem can be seen as assigning a value to a given cash
flow, the problem formulation nevertheless retains an emphasis on the ruin
time. This is contrasted by cash flow valuation principles such as real option
valuation, first introduced by Myers \cite{myers1977} in 1977. Whereas the
dividend problem seeks the value of a cash flow after passing through a
\emph{buffer} (the cash reserves), the real option approach evaluates a cash
flow without such a buffer. In other words, the latter is a valuation of a cash
flow without any liquidity constraint, as opposed to the optimal dividend
problem where the firm can reach ruin. The real option value thus provides a
natural bound for the optimal dividend value, which turns out to be helpful in
our analysis.

In the actuarial literature, the cash reserves are commonly described by a
spectrally negative Lévy process with a positive drift of premiums and negative
jumps of claims. Our direct focus is not an insurance firm, and we instead
study cash reserves described by a diffusion process. Although this is not the
natural insurance perspective, it is studied also there as the limiting case of
the jump processes, as initiated by \citet{iglehart1965}.

Formulated as a problem of `storage or inventory type', the general diffusion
problem with singular dividend policies was solved by Shreve et.\ al
\cite{shreve1984} in 1984. In the case of constant coefficients in the cash
reserves dynamics, Jeanblanc-Picqué \& Shiryaev \cite{jeanblanc1995} found the
solution by considering limits of solutions to problems with absolutely
continuous dividend strategies. The optimal solution to this singular problem
formulation is described by a so-called \emph{barrier strategy}, which yields a
reflected cash reserves process by paying any excess reserves as dividends.
This divides the state space into two regions: dividends are paid above the
barrier (\emph{dividend region}), but not in the region between zero and the
barrier (\emph{no-dividend region}). This is contrasted by dividend \emph{band
strategies} which frequently appear in jump models and were first identified by
Gerber \cite{gerber1969entscheidungskriterien}. Instead of the two spatial
regions, there then exists at least one no-dividend region in which the origin
is not contained. It thus creates a band-shaped no-dividend region in-between
two dividend regions.

In financial and economics literature, the main focus is on diffusion models,
and extensions often involve nonconstant interest rate, drift and/or diffusion
coefficients. Indeed, external, macroeconomic conditions and their effects on
profitability have a substantial impact on dividend policies, as shown by
\citet{gertler1991corporate} and more recently by \citet{hackbarth2006capital}.
Such macroeconomic effects have been studied in various forms. In particular,
\citet{anderson2012corporate} as well as \citet{barth2016non} numerically study
continuously changing stochastic parameters, whereas
\citet{akyildirim2014optimal} consider stochastic interest rates following a
Markov chain, and \citet{jiang2012optimal} consider model coefficients and
interest rate both governed by Markov chains. \citet{bolton2013market}
similarly study the macroeconomic impact on both financial and investment
opportunities. In contrast to coefficients influenced by macroeconomic factors,
\citet{radner1996} already in 1996 modelled a firm which alternates between
different operating strategies, thereby effectively controlling the model
coefficients. Other extensions include transaction costs of dividend payments
or the possibility of equity issuance, cf.\ \cite{akyildirim2014optimal,
bolton2013market, decamps2011free}. The papers by \citet{cai2006optimal}, and
\citet{cadenillas2007optimal} both treat different models with mean-reverting
cash reserves, contrasting the model here, where we instead consider
mean-reverting profitability. Finally,  \citet{avanzi2012mean} and
\citet{albrecher2017risk} propose dividend processes proportional and affine in
the cash reserves respectively as a means to capture the more stable dividend
streams seen in practise. For further references, we refer the reader to
\cite{asmussen2010ruin, albrecher2009optimality} and the references therein.

Our choice of diffusion model has a continuous, stochastic drift generated by a
separate profitability process. This structure yields a two-dimensional
problem in which the dividend strategy depends on the current profitability.
In particular, for low (negative) rates, a band strategy is optimal, but at
higher rates, dividends are optimally paid according to a barrier strategy,
with a barrier level depending on the profitability. Additionally, for very
low rates, we prove that it is optimal to perform a voluntary liquidation,
meaning that all cash reserves are paid instantaneously. Band structures are
common for jump models, but here appear in a continuous model.\footnote{Similar
	properties have been observed by \citet{anderson2012corporate} and
\citet{murto2014exit}.} Finally, in addition to qualitative and numerical
results, we provide proofs for continuity of the value function as well as a
comparison principle for the dynamic programming equation.

The paper is organized as follows: We begin by describing the problem in Section \ref{sec:dynamics}, followed by our assumptions and analytical results in Section \ref{sec:technical_statements}. In Section \ref{sec:numerics}, we present a numerical algorithm as well as its results. Intuition for these results are then provided by studying a related, simpler problem in Section \ref{sec:deterministic}. Finally, we suggest and numerically study some possible extensions of the model in Section \ref{sec:extensions}. At the end of the paper, after the concluding comments in Section \ref{sec:conclusion}, the proofs of the statements in Section \ref{sec:technical_statements} are given in Section \ref{sec:technical_proofs}.

\section{Problem formulation}

\label{sec:dynamics}
Consider a \emph{cash flow} on the form
\[
    \dif{C}^{\mu}_t = \mu^{\mu}_t \dif{t} + \xsigma \dif{W}_t, \quad C^{\mu}_0 = 0,
\]
where $W$ is a Brownian motion and the \emph{profitability} (cash flow rate) $\mu^\mu_t$ is described by
\[
    \dif{\mu}^{\mu}_t = \mudrift (\mu^\mu_t) \dif{t} + \musigma(\mu^\mu_t) \dif{\tilde{W_t}}, \quad \mu^{\mu}_0 = {\mu},
\]
for some functions $\mudrift$ and $\musigma$, as well as another Brownian
motion $\muW$ with correlation $\rho \in [-1,1]$ to $W$. Despite the formulation of
$\mu^\mu_t$ as a continuous process, most of the results extend naturally to the
Markov chains studied in the literature.

The precise assumptions on the diffusion, given in Assumption \ref{ass:regularity}, include Ornstein--Uhlenbeck processes
\[  \dif{\mu}^\mu_t = k (\mumean - \mu^\mu_t) \dif{t} + \musigma \dif{\muW}_t,  \]
for constants $k > 0$, $\mumean$, and $\musigma$ as well as another commonly considered process, the Cox--Ingersoll--Ross (CIR) process:
\[  \dif{\mu}^\mu_t = k (\mumean - \mu^\mu_t) \dif{t} + \musigma \sqrt{\mu^\mu_t - a} \dif{\muW}_t,  \]
for constants $k > 0$, $\mumean$, $\musigma$, and $a$. In fact, the Assumption \ref{ass:regularity} only imposes asymptotic conditions as $|\mu| \to \infty$. This means that on any given bounded domain, $\mudrift$ and $\musigma$ can be general, provided certain growth conditions are satisfied outside the bounded domain, and provided the SDE has a well-defined solution. This is naturally satisfied by bounded processes. Interpreting $-\mudrift$ as the derivative of some potential, it also includes the possibility of potentials with multiple wells (local minima), thus having several points of attraction.

We model a firm whose cash flow is given by the process $C^{\mu} = \process{C^{\mu}}$. The firm pays dividends to its shareholders using cash accumulated from the cash flow. Let $L_t$ denote the \emph{cumulative dividends} paid out until time $t$. Then the cash reserves $X = \process{X}$ of a firm with initial cash level $x$ can be written as
\[
    \dif{X}^{(x,\mu), L}_t = \dif{C}^{\mu}_t - \dif{L}_t, \quad X^{(x,{\mu}), L}_0 = x.
\]
The objective of the firm is to maximize its shareholders' value, defined as the expected present value of future dividends, computed under the risk-adjusted measure.\footnote{We assume that shareholders can diversify their portfolios and that the firm under study is small, so that its decisions do not alter the risk-adjusted measure.} Denote by $\muspace$ the domain on which $\mu_t$ resides. This domain is typically the whole real line, as for a Ornstein--Uhlenbeck process, a half-line, for a CIR process, or a bounded interval, for a bounded process. The value function is then defined as
\begin{equation}
    V(x, {\mu}) := \sup_{L} \mathbb{E} \left[ \int_0^{\theta^{(x,{\mu})}(L)} e^{-rt} \dif{L}_t \right]\!\!, \quad (x,{\mu}) \in \xmuspace := [0,\infty) \times \muspace,
	\label{eqn:dividendproblem}
\end{equation}
	where $L = \process{L}$ is required to be càdlàg, adapted to the filtration $\process{\mathcal{F}}$ generated by $C$ and $\mu$, as well as nondecreasing\footnote{The process $L$ must be nondecreasing because the limited liability of shareholders implies that $\dif{L}_t$ cannot be negative. Section \ref{sec:extensions} considers the case when new shares can be issued at a cost, allowing to inject new cash into the firm.} with $\Delta L_t \le X_{t-}$,\footnote{Shareholders cannot distribute more dividends than available cash reserves. Otherwise this would constitute fraudulent bankruptcy.}\footnote{Although the model allows dividend payments with infinite `frequency' of very general type, we argue that it is not less realistic than the absolutely continuous case where the `frequency' is also infinite, but interpreted as a rate. Indeed, as suggested by \eqref{eqn:dpe} in Section \ref{sec:dpe} and exploited in Section \ref{sec:numerics}, this model can be considered the limit when there is no bound on the dividend rate.} and
\[
	\theta^{(x,{\mu})}(L) := \inf \{ t > 0 : X^{(x,{\mu}),L}_t < 0 \}
\]
is the time of bankruptcy. In particular, we interpret a payout $\Delta L_t = X_{t-}$ as a decision to liquidate the firm.

In this paper we prove that liquidation is optimal when the profitability falls below a certain level; that, by stochastic methods, the value function is continuous and is a viscosity solution of the dynamic programming equation (DPE); that the DPE satisfies the comparison principle; and finally we provide a numerical scheme and extensive numerical results. Thereafter, we describe a related deterministic problem whose solution turns out to share most of the structural properties, thus providing intuition for the original problem. Finally, we conduct numerical studies of some model extensions.

When the starting points $(x,\mu)$ of the cash reserves and the cash flow are clear from context, the superscripts will be dropped in order to simplify the notation. Similar omissions of superscripts will be done for the bankruptcy times and strategies $L$ when it is clear what dividend policy is followed.

\section{Main results}
\label{sec:technical_statements}
The state space can be divided into regions, each characterized by its role in the DPE---interpreted as a control action. An illustration of these regions is presented in Figure \ref{fig:regions}.
The value function is characterized by three main regions: the dividend region, retain earnings region, and the liquidation region. The region of retained earnings is bounded by two curves and is characterized by $\dif{L}_t = 0$. The dividend region and liquidation region are both characterized by $\dif{L} \neq 0$, but correspond to different interpretations, and intersect at the threshold $\muthresh$. More precisely, in the liquidation region, all available cash reserves are `paid', leading to a liquidation. This is in contrast to the dividend region, where only the excess of the \emph{dividend target} is paid to the shareholders.
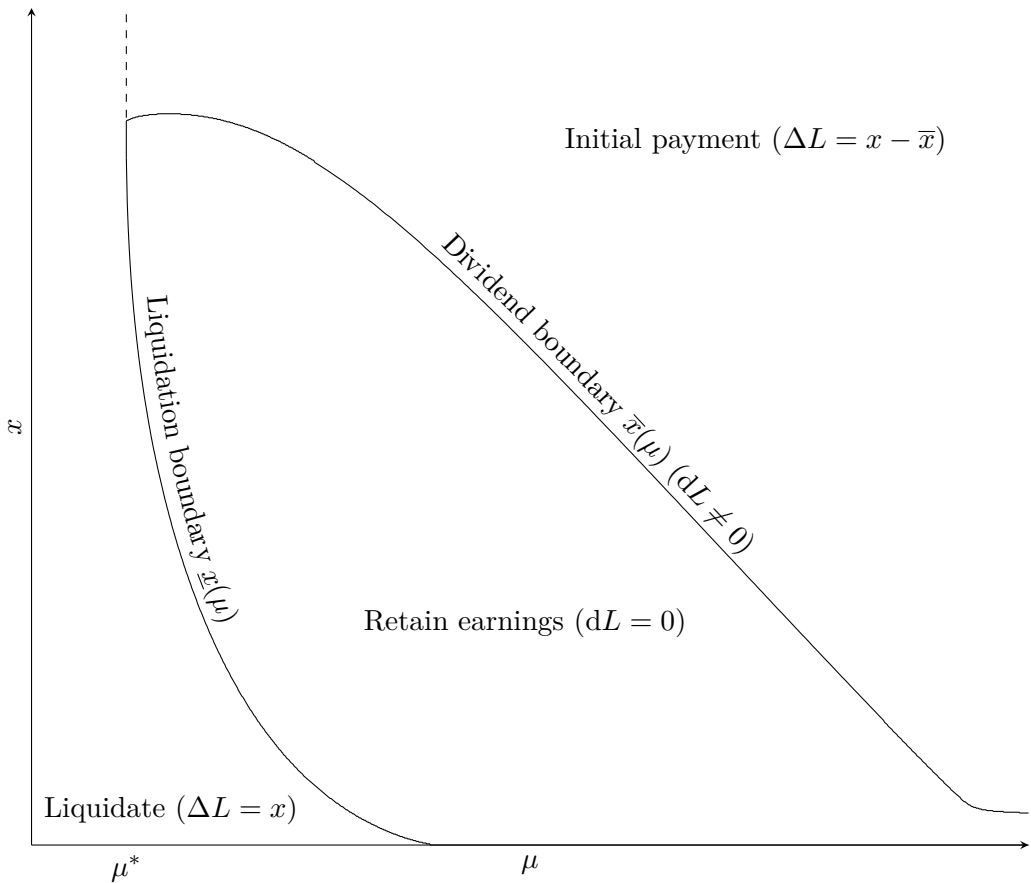
\begin{figure}
    \centering
	\begin{tikzpicture}
	    \begin{axis}[name=axis, xlabel=$\mu$, ylabel=$x$, width=\textwidth,  no markers, ticks=none, axis x line=bottom, axis y line=left, enlarge x limits={abs value=0.2,lower}, enlarge y limits={abs value=0.3,upper}]
		\addplot[name path=divLower, postaction={decorate,decoration={text along path,transform={yshift=0.5em},text align={left indent={0.1\dimexpr\pgfdecoratedpathlength\relax}},text={Liquidation boundary {$\underline{x}(\mu)$}}}}] table [x index=0, y index=1, col sep=comma] {Figures/xmusigma010030smooth.csv} node[pos=0.0, inner sep=0pt] (muminLow) {};
		\addplot[name path=divUpper, postaction={decorate,decoration={text along path,transform={yshift=0.4em},text align={left indent={0.3\dimexpr\pgfdecoratedpathlength\relax}},text={Dividend boundary {$\overline{x}(\mu)$} ({$\dif{L} \neq 0$})}}}] table [x index=0, y index=2, col sep=comma] {Figures/xmusigma010030smooth.csv} node[pos=0.0, inner sep=0pt] (muminHigh) {};
		\draw let \p1 = (muminLow), \p2 = (muminHigh) in (\x1,\y1) -- (\x1, \y2);
		\draw[dashed] let \p1 = (muminHigh) in (\x1,\y1) ++(0cm,5cm) -- (\x1,\y1);
		\node[below left] at (0.3,0.7) {Retain earnings ($\dif{L} = 0$)};
		\node[left, rotate=0] at (-0.52,0.10) {Liquidate ($\Delta L = x$)};
		\node[right] at (0,2) {Initial payment ($\Delta L = x - \overline{x}$)};
	    \end{axis}
	    \path let \p1 = (muminLow) in (\x1,\y1)-- (\x1,0cm) node[below] {$\muthresh$};
	\end{tikzpicture}
	\caption{\label{fig:regions}The figure shows the three regimes. In the region between the lines $\underline{x}$ and $\overline{x}$ the all incoming profits are retained. When the cash reserves fall to $\underline{x}$, the firm liquidates, whereas when it increases to $\overline{x}$, dividends are paid out according to the local time at the boundary, thus reflecting the cash reserves process. In the region above $\overline{x}$, a lump sum of the excess of $\overline{x}$ is paid immediately. Finally, when $\mu \leq \muthresh$, liquidation is optimal at all cash levels.}
\end{figure}

The remainder of this section is devoted to the statements of the main results. The proof of these results are given in Section \ref{sec:technical_proofs}. For the proofs we need the following set of assumptions, satisfied by for example Ornstein--Uhlenbeck and CIR processes. This assumption essentially says that the profitability process is mean-reverting and well-behaved in the sense of a Feller condition. Mean-reversion of profitability is empirically well established by \citet{fama_forecasting_2000} and others.

\begin{assumption}
    \label{ass:regularity}
    Throughout, we assume that the domain of $\process{\mu}$ is some possibly unbounded interval $\muspace$. Moreover, we assume that $\mudrift$ and $\musigma^2$ are locally Lipschitz continuous on the interior $\interior{\muspace}$, that $-\mu/\mudrift$ is non-negative and bounded for large (positive) $\mu$, that $-\mudrift/\mu$ is non-negative and bounded for large $-\mu$, as well as that $\musigma^2 \in \bigO(\mu)$ and never vanishes in $\interior{\muspace}$. Finally, we also assume either of the following:
    \begin{enumerate}
	\item The function $\musigma^2$ is also locally Lipschitz on the boundary $\bdry\muspace$.
	\item For any sufficiently small $\eta > 0$, we assume that for $\nu = \inf \muspace$,
	    \[  \limsup_{\mu \to \nu} \left( \frac{1}{\mu - \nu} - \frac{2 \mudrift(\mu) + \eta \rho \xsigma \musigma(\mu)}{\musigma(\mu)^2} \right) < \infty, \]
	    and for $\nu = \sup \muspace$
	    \[  \liminf_{\mu \to \nu} \left( \frac{1}{\mu - \nu} - \frac{2 \mudrift(\mu) + \eta \rho \xsigma \musigma(\mu)}{\musigma(\mu)^2} \right) > - \infty. \]
    \end{enumerate}
\end{assumption}

The economic interpretation of the growth conditions on $\mudrift$ is that even if the profitability is very large, it eventually returns to a more reasonable level. The growth condition on $\musigma$ simply ensures that the diffusion does not overpower this effect. Finally, the $\limsup$ and $\liminf$ conditions at the boundary points are needed to ensure that the profitability process behaves well enough close to the boundary.

\subsection{Liquidation threshold}
\begin{restatable}{theorem}{liquidationThreshold}
	\label{thm:liquidationThreshold}
	If $\muspace$ has no lower bound, there exists a value $\muthresh$ such that it is optimal to liquidate immediately whenever $\mu \leq \muthresh$, i.e.\ $V(x, \mu) \equiv x$.
\end{restatable}

\subsection{Continuity}
Although the problem formulation bears resemblance to the Merton consumption problem, which has been extensively studied in the mathematical finance literature, the crucial difference here is that the firm is always exposed to the risk of its own operations. In other words, there is no entirely safe asset, and, as a result, the problem lacks desired concavity properties. More specifically, this happens due to the possible quasi-convexity of $L \mapsto \theta(L)$. Figure \ref{fig:nonconvex} shows a case where $\theta\left(\frac{L^1 + L^2}{2}\right) < \max \{ \theta(L^1), \theta(L^2) \}$ for two strategies $L^1$ and $L^2$, which means that the convex combination of strategies in some scenario would pay out dividends after bankruptcy. Note that this loss of concavity corresponds to nonconvexity of the set of dividend processes that are constant after ruin.

\begin{figure}
    \centering
    \newcommand{\brownianmotion}[5]{
\draw[name path=brownianmotion, cap=round, #5] (#2)
\foreach \x in {1,...,#1}
{   -- ++(#3,rand*#4) coordinate (current point)
}
}

\begin{tikzpicture}[scale = 1.0, cap=round]
	\draw[->] (0,0) -- (10,0) node [right] {$t$};
	\draw[->] (0, 0) -- (0, 4.5) node [midway, above, sloped] {Cash and dividends};
\pgfmathsetseed{666}
	\brownianmotion{100}{0,3}{0.03}{0.1}{black};
	\node [above] at (1.7,3.15) {$x + C_t(\omega)$};
	\coordinate (firstruin) at (current point);
	\path[name path=firstpayout] (0,0) -- (firstruin);
	\draw[name intersections={of=firstpayout and brownianmotion, by={intersect}}] (0,0) -- node[left] {$L^1_t\,$} (intersect) coordinate (firstbankrupt);
	\draw let \p1 = (firstbankrupt) in (\x1, 0.1) -- (\x1, -0.1) node[below] {$\theta(L^1)$};

    \pgfmathsetseed{66666}
    \brownianmotion{170}{current point}{0.03}{0.1}{black};
    \coordinate (secondbankrupt) at (current point);

    \newcommand{\timeofsecondpayout}{4.0}
    \path[name path=secondpayout] (\timeofsecondpayout,0) -- (secondbankrupt);
    \draw[name intersections={of=secondpayout and brownianmotion, by={secondbankrupt}}] (\timeofsecondpayout,0) -- (secondbankrupt) node[midway, above] {$L^2_t$};
    \draw let \p1 = (secondbankrupt) in (\x1, 0.1) -- (\x1, -0.1) node[below] {$\theta(L^2)$};

    \path let \p1 = (firstruin) in coordinate (halffirst) at ($(firstruin)!0.5!(\x1,0)$);
    \path let \p1 = (halffirst) in coordinate (middlesecondstart) at (\timeofsecondpayout, \y1);
    \path let \p1 = (middlesecondstart), \p2 = (secondbankrupt) in coordinate (middlesecond) at ($(0, \y1) + (\x2,0)!0.5!(secondbankrupt)$);
    \path let \p1 = (middlesecond) in coordinate (middlefinal) at (10, \y1);
    \draw[dashed, name path=middlesecondincrease] let \p1 = (middlesecond) in (0,0) -- (halffirst) -- (middlesecondstart) -- (middlesecond) -- (middlefinal);
    \draw[name intersections={of=middlesecondincrease and brownianmotion, by={middleruin}}] let \p1 = (middleruin) in (\x1, 0.1) -- (\x1, -0.1) node[below] {$\theta(\bar L_t)$};

    \brownianmotion{30}{current point}{0.03}{0.1}{black};
    \draw let \p1 = (firstbankrupt) in (firstbankrupt) -- (10, \y1);
    \draw let \p1 = (secondbankrupt) in (secondbankrupt) -- (10, \y1);
\end{tikzpicture}
    \caption{\label{fig:nonconvex} Recall that ruin occurs when the dividends reach the total cash accumulated, i.e.\ $x + C_t$. The figure illustrates two dividend policies $L^1$ and $L^2$ (solid) as well as their convex combination $\bar{L}=(L^1+L^2)/2$ (dashed), showing that for some path $\theta(\bar{L}) < \theta(L^1) \vee \theta(L^2)$, which corresponds to the possibility of dividend payments after ruin. Equivalently, the set of dividend strategies which are constant after ruin is nonconvex.}
\end{figure}
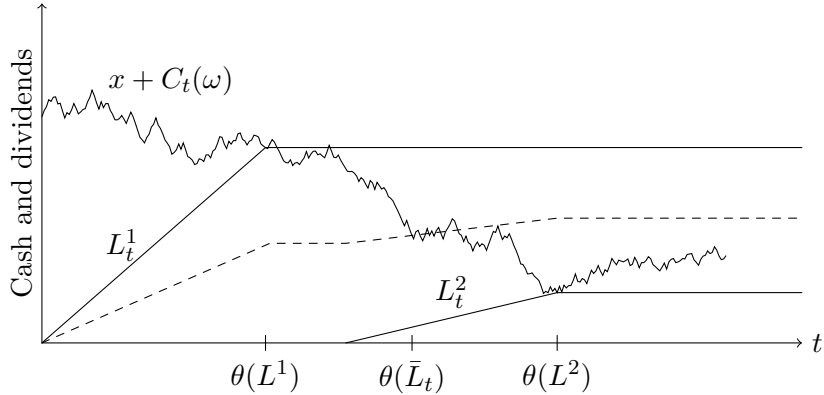
This lack of concavity is the reason for our necessity to prove continuity, given by the following theorem:
\begin{restatable}{theorem}{continuity}
	\label{thm:continuity}
	The value function is continuous everywhere.
\end{restatable}

\subsection{Dynamic programming equation}
\label{sec:dpe}
Following the continuity of Theorem \ref{thm:continuity}, we refer to \cite{fleming2006controlled} for proving the dynamic programming principle. For a general proof of dynamic programming, we refer to \cite{karoui2013capacitiesI,karoui2013capacitiesII}. Writing
\[  \mathcal{L}V = \mu V_x + \kappa(\mu) V_\mu + \Tr \Sigma(\mu) D^2 V, \]
where
\[\Sigma = \begin{bmatrix}
    \xsigma^2 & \rho \xsigma \musigma \\
    \rho \xsigma \musigma & \musigma^2
\end{bmatrix}\]
is the covariance matrix, the dynamic programming equation corresponding to \eqref{eqn:dividendproblem} is given by
\begin{equation} \label{eqn:dpe}
    \min \{ rV - \mathcal{L}V, V_x - 1 \} = 0, \quad \text{ in } \mathbb{R}_{> 0} \times \muspace,
\end{equation}
with $V(0, \cdot) \equiv 0$.

\begin{restatable}[Comparison]{theorem}{comparison}
    Let $u$ and $v$ be upper and lower semicontinuous, polynomially growing viscosity sub- and supersolutions of \eqref{eqn:dpe}. Then $u \leq v$ for $x = 0$ implies that $u \leq v$ everywhere in $\mathcal{O} := \mathbb{R}_{\geq 0} \times \muspace$.
    \label{thm:comparison}
\end{restatable}

\begin{corollary}[Uniqueness]
    \label{cor:uniqueness}
    The value function is the unique subexponentially growing \emph{viscosity solution} to the dynamic programming equation \eqref{eqn:dpe}.
\end{corollary}
\begin{proof}
    By the dynamic programming principle, the value function $V$ is a solution to \eqref{eqn:dpe}. To obtain uniqueness, observe that, by Theorem \ref{thm:comparison}, $V$, being both a sub- and a supersolution, both dominates and is dominated by any other solution. In other words, it is equal to any other solution, and thus unique.
\end{proof}

Note that the importance of the comparison principle goes beyond the uniqueness of the solution; the principle underpins the stability property of viscosity solutions. The stability property, in turn, leads to convergence of numerical schemes to the (unique) solution \cite{barles1991convergence}.

\section{Numerical results}
\label{sec:numerics}
The numerical results presented in this section are all obtained through policy iteration. Policy iteration is an iterative technique where one chooses a policy/control, calculates the corresponding payoff function, then updates the policy where the payoff function suggests it is profitable, and finally iterates this procedure until convergence. That the scheme does indeed converge to the value function is supported by the comparison principle in Theorem \ref{thm:comparison} and the uniqueness result in Corollary \ref{cor:uniqueness}. 

The idea implemented here is to approximate the singularity with increasingly large controls which are absolutely continuous with respect to time. In particular, let $K > 0$ be any large constant and consider control variables $L_t = \int_0^t \ell(X^L_t) \dif{t}$, where $\ell(x) \in [0,K]$ is measurable. Then, the problem of optimizing over functions $\ell$ amounts to a penalization of the DPE \eqref{eqn:dpe} with penalization factor $K$.

To see that the limit of these problems gives the solution to the original problem, we begin by writing out the PDE for the approximation:
\begin{equation}
    \label{eqn:dpediscretized}
    \min_{\ell \in [0,K]} \left(rV^K - \mathcal{L}V^K\right) + \ell \left(V^K_x - 1\right) = 0.
\end{equation}
By dividing by $K$ and subtracting and adding equal terms, we reach
\[  \min_{\lambda \in [0,1]} (1 - \lambda)\frac{rV^K - \mathcal{L}V^K}{K} + \lambda \left(V^K_x - 1 + \frac{rV^K - \mathcal{L}V^K}{K} \right) = 0, \]
which is equivalent to
\[  \min \left\{ rV^K - \mathcal{L}V^K, V^K_x - 1 + \frac{rV^K - \mathcal{L}V^K}{K} \right\} = 0.  \]
Finally, letting $K \to \infty$ and using the stability property of viscosity solutions guaranteed by the comparison principle, we find that $V^K \to V$.

Thanks to this approximation of the state dynamics, it holds that in any given space discretization, the transition rate between states is bounded away from zero. This means that the continuous time Markov chain on the discretized space can be reduced to a discrete time Markov chain (cf., e.g., \cite{puterman1994markov}). Thus, after a suitable space discretization, the problem is solved using standard methods of policy iteration that are known to converge. The convergence to $V^K$ of the solution to the discretized problem can be also be shown with standard viscosity methods.

\newcommand{\Dc}{\mathcal{D}}%
\newcommand{\Lc}{\mathcal{L}_{\mathcal{D}}}%
Let $\Dc$ be a discretization of $\xmuspace$ consisting of $N$ points, and let $\Lc^\ell$ be a corresponding discretization of $rV^K - \mathcal{L}V^K - \ell V^K_x$ from \eqref{eqn:dpediscretized}. Then, starting with any control $\ell_0$, the policy iteration scheme with tolerance $\tau \geq 0$\footnote{Note that the policy iteration scheme halts even for $\tau = 0$.} is given by the following steps:\footnote{Since we solve in a bounded domain, some care has to be taken at the boundaries. However, thanks to the condition given on $\mudrift$, it is natural to impose a reflecting boundary along the $\mu$-directions, provided the domain is large enough. Moreover, the optimal policy will naturally reflect at the upper $x$-boundary, provided the domain is large enough to contain the no-dividend region. For these reasons, the precise choice of boundary condition is of relatively small importance, if $\Dc$ is chosen appropriately.}
\noindent\begin{minipage}{\textwidth}
\rule{\textwidth}{1pt}
\noindent\textbf{Policy iteration algorithm} (step $i$)\\
\rule{\textwidth}{1pt}
\begin{enumerate}[topsep=0pt]
	\item Compute $V_i \in \R^{N}$ such that
		\[
		    \sum_{(x', \mu') \in \Dc} \Lc^{\ell_i} (x,\mu, x',\mu') V_i(x', \mu') + \ell_i = 0, \quad \forall (x,\mu) \in \Dc.
		\] Halt if $\sup |V_{i} - V_{i-1}|\leq \tau$.

	    \item For each $(x,\mu) \in \Dc,$ compute $\ell_{i+1}(x,\mu)$ according to
		\[
		    \ell_{i+1}(x,\mu) \in \argmin_{\hat{\ell} \in [0,K]} \left( \sum_{(x',\mu') \in \Dc} \Lc^{\hat{\ell}} (x,\mu, x', \mu') V_i(x', \mu') + \hat{\ell} \right).
		\]
		\item Return to step (i).
\end{enumerate}
\vspace{-\parskip}
\rule{\textwidth}{1pt}
\end{minipage}

\begin{remark}
	In the algorithm above, the second step is computationally the most difficult. In general, there is no known structure to be used, so it often has to be solved by brute force. Luckily, however, the problems given for each $(x,\mu)$ are fully independent, and the step can thus be entirely parallelized. Thanks to this, the computational burden can be partly mitigated.
\end{remark}

\subsection{Results and comparative statics}

What follows is an account of an implementation of the scheme for the Ornstein--Uhlenbeck model
\[  \dif{\mu}_t = k(\mumean - \mu) \dif{t} + \musigma \dif{\muW}_t. \]
The resulting optimal strategies presented in Figure \ref{fig:basicmodel}, with $k = 0.5$, $\mumean = 0.15$, and $\musigma = 0.1$ (left) as well as $\musigma = 0.3$ (right). The other parameter choices are $\sigma = 0.1$, $\rho = 0$, and $r = 0.05$.
The white regions indicate dividend payments or liquidation, i.e., $V_x = 1$ and $\dif{L} > 0$, whereas the black regions indicate that the firm retains all its earnings, i.e., $rV - \mathcal{L}V = 0$ and $\dif{L} = 0$.
The figures show the optimal policy, from which the value function can be obtained by solving a linear system of equations.

\begin{figure}
    \centering
	\begin{tikzpicture}
	    \begin{axis}[name=axis, xlabel=$\mu$, width=0.5\textwidth, enlargelimits=false, enlarge x limits={abs value=0.05,lower}, enlarge y limits={abs value=0.05,upper}, no markers, legend pos=south east, legend cell align=left]
		\addplot[name path=divLower] table [x=mu, y=divLower, col sep=comma] {Figures/xmusigma010030.csv};
		\addplot[name path=divUpper] table [x=mu, y=divUpper, col sep=comma] {Figures/xmusigma010030.csv};
		\addplot[fill=black] fill between[of=divLower and divUpper];
	    \end{axis}
	    \node[draw=black, anchor=north east, below left=2mm] at (axis.north east) {%
		    \begin{tabular}{@{}r@{ }l@{}}
			\tikz{\draw[black] (0,0) rectangle (3mm, 3mm);}&Pay dividends \\
			    \tikz{\filldraw[draw=black, fill=black] (0,0) rectangle (3mm, 3mm);}&No dividends
			\end{tabular}};
	\end{tikzpicture}
	\begin{tikzpicture}
	    \begin{axis}[name=axis, xlabel=$\mu$, ylabel=$x$, width=0.5\textwidth, enlargelimits=false, enlarge x limits={abs value=0.05,lower}, enlarge y limits={abs value=0.05,upper}, no markers, legend pos=south east, legend cell align=left]
		\addplot[name path=divLower] table [x index=0, y index=1, col sep=comma] {Figures/musigma010.csv};
		\addplot[name path=divUpper] table [x index=0, y index=2, col sep=comma] {Figures/musigma010.csv};
		\addplot[fill=black] fill between[of=divLower and divUpper];
	    \end{axis}
	    \node[draw=black, anchor=north east, below left=2mm] at (axis.north east) {%
		    \begin{tabular}{@{}r@{ }l@{}}
			\tikz{\draw[black] (0,0) rectangle (3mm, 3mm);}&Pay dividends \\
			    \tikz{\filldraw[draw=black, fill=black] (0,0) rectangle (3mm, 3mm);}&No dividends
			\end{tabular}};
	\end{tikzpicture}
	\caption{\label{fig:basicmodel}The black region corresponds to $\dif{L} = 0$, whereas the white region corresponds to $\dif{L} > 0$. The latter case is interpreted as either dividend payments or liquidation, depending on the position in the state space, see Figure \ref{fig:regions}.}
\end{figure}
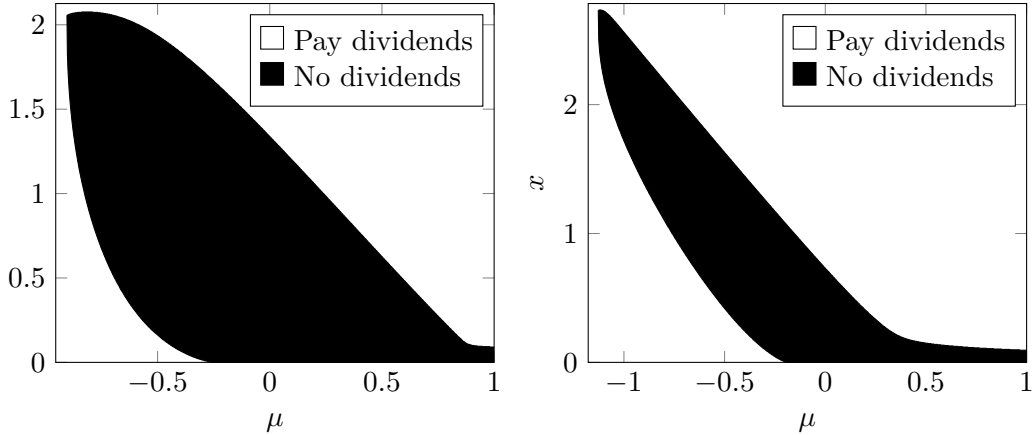

Figure \ref{fig:comparativestatics} shows the effect of changing one parameter at a time. The figures presented here have been generated for $\tau = 0$, without parallelization on an ultrabook laptop.\footnote{Intel Core i7-7600U, 1866MHz LPDDR3.} On this system, the computations took around 10 minutes\footnote{For the left solution in Figure \ref{fig:basicmodel} (same as `default' values from Figure \ref{fig:comparativestatics}), the runtime was just short of 8 minutes on a $1000 \times 1000$ grid without parallelization. A more modest grid size of $300 \times 300$ runs in 7 seconds on the same hardware. All computations were generated with a starting point with `no information', i.e., $\ell$ identically zero.} at this resolution, and were observed to converge.
Varying the parameters does not seem to change the qualitative properties significantly. The parameter $\musigma$ primarily changes the width of the band region; $k$ and $\mumean$ affect the size and extension into the region of negative $\mu$; $\xsigma$ changes the height; and finally $\rho$ influences the shape. Note that although the free boundary is nonmonotone in $\musigma$ for $x$ right below 2, it is monotone for smaller $x$.

\begin{figure}
    \centering
    \begin{tikzpicture}
	\begin{axis}[name=axis, xlabel=$\mu$, ylabel=$x$, width=0.5\textwidth, enlargelimits=false, xmax=1.0, enlarge x limits={abs value=0.05,lower}, enlarge y limits={abs value=0.05,upper}, no markers, legend pos=south east, legend cell align=left, font=\footnotesize]
	    \addplot[black] table [x index=0, y index=1, col sep=comma] {Figures/musigma010.csv} node[pos=0.6, right] {$\musigma = 0.1$};
	    \addplot[black] table [x index=0, y index=2, col sep=comma] {Figures/musigma010.csv};
	    \addplot[black] table [x index=0, y index=1, col sep=comma] {Figures/musigma020.csv};
	    \addplot[black] table [x index=0, y index=2, col sep=comma] {Figures/musigma020.csv};
	    \addplot[black] table [x index=0, y index=1, col sep=comma] {Figures/xmusigma010030.csv};
	    \addplot[black] table [x index=0, y index=2, col sep=comma] {Figures/xmusigma010030.csv};
	    \addplot[black] table [x index=0, y index=1, col sep=comma] {Figures/musigma040.csv};
	    \addplot[black] table [x index=0, y index=2, col sep=comma] {Figures/musigma040.csv};
	    \addplot[black] table [x index=0, y index=1, col sep=comma] {Figures/musigma050.csv} node[pos=0.4, left, inner sep=0.5em] {$\musigma = 0.5$};
	    \addplot[black] table [x index=0, y index=2, col sep=comma] {Figures/musigma050.csv};
        \end{axis}
    \end{tikzpicture}
    \begin{tikzpicture}
	\begin{axis}[name=axis, ytick={0,1,2,3}, xlabel=$\mu$, width=0.5\textwidth, enlargelimits=false, enlarge x limits={abs value=0.05,lower}, enlarge y limits={abs value=0.05,upper}, no markers, legend pos=south east, legend cell align=left, font=\footnotesize]
	    \addplot[black] table [x index=0, y index=1, col sep=comma] {Figures/xmusigma010030.csv} node[pos=0.55, left] {$\xsigma = 0.1$};
	    \addplot[black] table [x index=0, y index=2, col sep=comma] {Figures/xmusigma010030.csv};
	    \addplot[black] table [x index=0, y index=1, col sep=comma] {Figures/xsigma020.csv};
	    \addplot[black] table [x index=0, y index=2, col sep=comma] {Figures/xsigma020.csv};
	    \addplot[black] table [x index=0, y index=1, col sep=comma] {Figures/xsigma030.csv};
	    \addplot[black] table [x index=0, y index=2, col sep=comma] {Figures/xsigma030.csv};
	    \addplot[black] table [x index=0, y index=1, col sep=comma] {Figures/xsigma040.csv};
	    \addplot[black] table [x index=0, y index=2, col sep=comma] {Figures/xsigma040.csv} node[pos=0.6, above right] {$\xsigma = 0.4$};
        \end{axis}
    \end{tikzpicture}
	\begin{tikzpicture}
	    \begin{axis}[name=axis, xlabel=$\mu$, ylabel=$x$, width=0.5\textwidth, xmax=1.0, enlargelimits=false, enlarge x limits={abs value=0.1,lower}, enlarge y limits={abs value=0.1,upper}, no markers, legend pos=south east, legend cell align=left, font=\footnotesize]
		\addplot[black] table [x index=0, y index=1, col sep=comma] {Figures/xmusigma010030.csv} node[pos=0., above] {$\mumean = 0.15$};
		\addplot[black] table [x index=0, y index=2, col sep=comma] {Figures/xmusigma010030.csv};
		\addplot[black] table [x index=0, y index=1, col sep=comma] {Figures/mubar0.csv};
		\addplot[black] table [x index=0, y index=2, col sep=comma] {Figures/mubar0.csv} node[pos=0.1, above] {$\mumean = 0$};
		\addplot[black] table [x index=0, y index=1, col sep=comma] {Figures/mubar030.csv};
		\addplot[black] table [x index=0, y index=2, col sep=comma] {Figures/mubar030.csv} node[pos=0.25, above right] {$\mumean = 0.3$};
	    \end{axis}
	\end{tikzpicture}
	\begin{tikzpicture}
	    \begin{axis}[name=axis, xlabel=$\mu$, width=0.5\textwidth, xmax=1.2, enlargelimits=false, enlarge x limits={abs value=0.1,lower}, enlarge y limits={abs value=0.1,upper}, no markers, legend pos=south east, legend cell align=left, font=\footnotesize]
		\addplot[black] table [x index=0, y index=1, col sep=comma] {Figures/xmusigma010030.csv};
		\addplot[black] table [x index=0, y index=2, col sep=comma] {Figures/xmusigma010030.csv} node[pos=0.1, above right] {$k = 0.5$};
		\addplot[black] table [x index=0, y index=1, col sep=comma] {Figures/k025.csv} node[pos=0.30, right] {$k=0.25$};
		\addplot[black] table [x index=0, y index=2, col sep=comma] {Figures/k025.csv};
		\addplot[black] table [x index=0, y index=1, col sep=comma] {Figures/k10.csv} node[pos=0.5, below left] {$k=1.0$};
		\addplot[black] table [x index=0, y index=2, col sep=comma] {Figures/k10.csv};
	    \end{axis}
	\end{tikzpicture}
	\begin{tikzpicture}
	    \begin{axis}[name=axis, xlabel=$\mu$, ylabel=$x$, width=0.5\textwidth, xmax=1.0, enlargelimits=false, enlarge x limits={abs value=0.1,lower}, enlarge y limits={abs value=0.1,upper}, no markers, legend pos=south east, legend cell align=left, font=\footnotesize]
		\addplot[black] table [unbounded coords=jump, x index=0, y index=1, col sep=comma] {Figures/rho-10.csv};
		\addplot[black] table [x index=0, y index=2, col sep=comma] {Figures/rho-10.csv} node[pos=0.45, below left] {$\rho = -1$};
		\addplot[black] table [x index=0, y index=1, col sep=comma] {Figures/rho-05.csv};
		\addplot[black] table [x index=0, y index=2, col sep=comma] {Figures/rho-05.csv};
		\addplot[black] table [x index=0, y index=1, col sep=comma] {Figures/xmusigma010030.csv};
		\addplot[black] table [x index=0, y index=2, col sep=comma] {Figures/xmusigma010030.csv};
		\addplot[black] table [x index=0, y index=1, col sep=comma] {Figures/rho05.csv};
		\addplot[black] table [x index=0, y index=2, col sep=comma] {Figures/rho05.csv};
		\addplot[black] table [x index=0, y index=1, col sep=comma] {Figures/rho10.csv};
		\addplot[black] table [x index=0, y index=2, col sep=comma] {Figures/rho10.csv} node[pos=0.55, above right] {$\rho = 1$};
	    \end{axis}
	\end{tikzpicture}
	\caption{\label{fig:comparativestatics} Comparative statics. Apart from for the parameter being varied, the chosen values were $r = 0.05$, $k = 0.5$, $\mumean = 0.15$, $\musigma = 0.3$, $\xsigma = 0.1$, and $\rho = 0$. The parameter varied is indicated in the respective figure. The values considered were $k = 0.25, 0.5, 1.0$, $\mumean = 0.0, 0.15, 0.3$, $\musigma = 0.1,0.2,0.3,0.4,0.5$, $\xsigma = 0.1,0.2,0.3,0.4$, and $\rho = -1.0,-0.5,0.0,0.5,1.0$. To address the effect of the boundary conditions, most calculations were run on a larger domain than plotted here. The lower boundary for $\rho=-1.0$ displayed signs of numerical instability around the origin and has therefore not been plotted in this region.}
\end{figure}
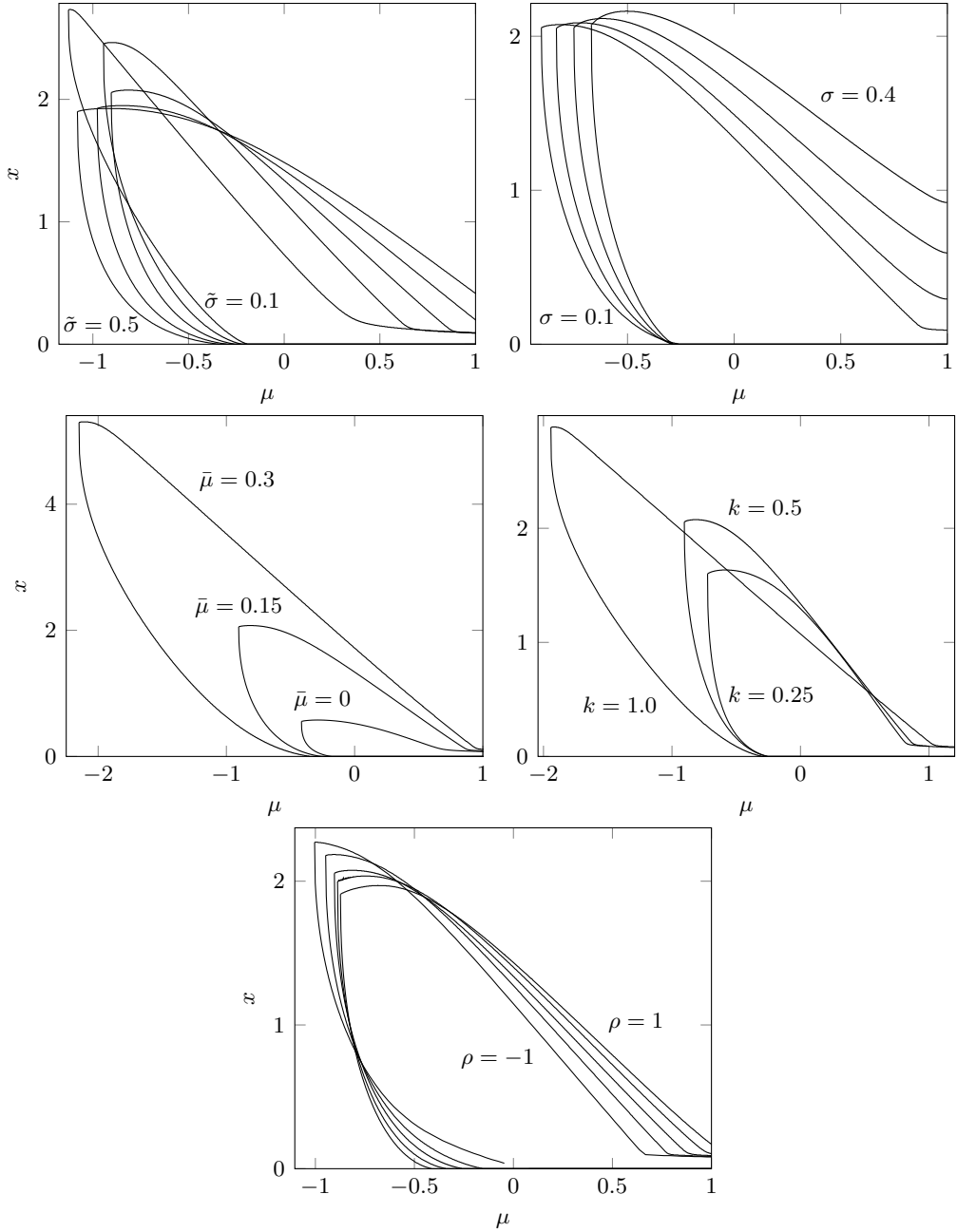

\section{A source of intuition: the deterministic problem}
\label{sec:deterministic}
WHen the two parameters $\sigma$ and $\tilde{\sigma}$ are zero, the problem can be solved explicitly for a mean-reverting $\mudrift(\mu) = k(\mumean - \mu)$, with $k,\mumean > 0$. The solution of this deterministic problem provides intuition for the solution of the stochastic problem.

Since $\process{\mu}$ is mean-reverting to the positive value $\mumean$, it will never attain negative values once it has been positive. In particular, solving the ODE describing the dynamics of $\process{\mu}$ yields
\[ \mu_t = \mumean + (\mu - \mumean) e^{-kt}.  \]
We will treat the cases $\mu \geq 0$ and $\mu < 0$ separately.

\begin{figure}
    \centering
    \begin{tikzpicture}
	\begin{axis}[enlargelimits=false, xlabel=$\mu$, legend pos=south west, width=\textwidth, height=7cm, xmin=-1.85, xmax = 0.65]
	    \addplot[name path=xb, restrict x to domain=-1.72:0] table {Figures/deterministic.txt};
	    \addlegendentry{$x_b(\mu)$}
	    \addplot[name path=futureincome, dashed, restrict x to domain=-1.72:0] table {Figures/deterministicfutureincome.txt};
	    \addlegendentry{$e^{-r\tau_0(\mu)}V(0,0)$}
	    \path[name intersections={of=xb and futureincome, by=cutoff}];
	    \path (cutoff) [dotted, draw=black] -- (cutoff|-0,7cm);
	    
	    \node[rotate=-32] at (-1.62, 1.8) {Liquidate};
	    \node at (-0.2, 2) {Pay excess of $x_b(\mu)$: $\displaystyle L_t = x-x_b(\mu) + \int_{\mathrlap{\tau_0(\mu)\wedge t}}^t \mu_s \dif{s}$};
	\end{axis}
    \end{tikzpicture}
    \caption{The value $x_b(\mu)$ is the cost of waiting for positive cash flow, whereas $e^{-r\tau_0(\mu)}V(0,0)$ is the present value of the future positive cash flow. For $\mu$ below the level at which these coincide, liquidation is thus optimal. Liquidation is also optimal when $x < x_b(\mu)$.}
\end{figure}
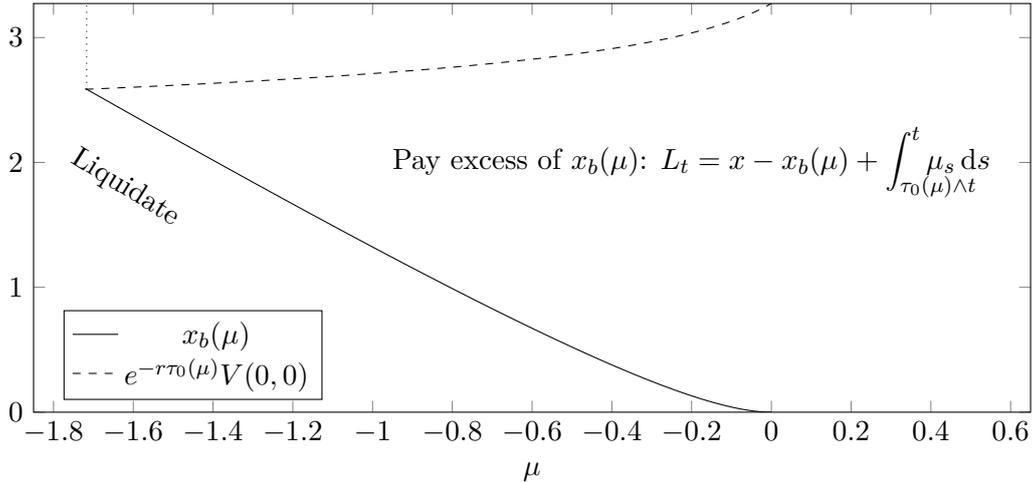

If $\mu \geq 0$, the firm is always profitable, and does not default at $x = 0$ unless $\mu < 0$. Therefore, there is no need for a cash buffer, so it is optimal to pay all the initial cash reserves immediately. Thereafter, cash from $\mu$ is paid out as it flows in. The value of paying the incoming earnings as dividends is
\begin{equation}
	\label{eqn:longruncashflow}
	\int_0^\infty e^{-r t} \mu_t \dif{t} = \int_0^\infty \mumean e^{-rt} + (\mu - \mumean) e^{-(k + r)t} \dif{t} = \frac{\mumean}{r} + \frac{\mu - \mumean}{r + k}.
\end{equation}
Hence, for $\mu \geq 0$, the value is given by
\[  V(x, \mu) = x + \frac{\mumean}{r} + \frac{\mu - \mumean}{r + k}. \]

On the other hand, if the cash flow starts at a negative level, it will eventually reach a positive state, but the question is whether the firm is able to absorb the cumulated losses before then. If it can, are those losses larger than future earnings?
More precisely, the company could face ruin before it sees positive earnings, but even if it does not, the losses incurred could offset the value of the future positive cash flows.
To address the first possibility, we calculate the minimum amount of cash needed to reach a positive cash flow before the time of ruin. Denote by $\tau_0$ the time such that $\mu_{\tau_0} = 0$. This time can be found explicitly:
\[
	\tau_0 = \tau_0(\mu) = \frac{\ln \left(\frac{\mumean}{\mumean - \mu} \right)}{-k}.
\]
The cumulative losses until a positive cash flow is reached are:
\[
	\int_0^{\tau_0(\mu)} \mu_t \dif{t} = \mumean \tau_0(\mu) + \frac{\mu - \mumean}{k} (1 - e^{-k \tau_0(\mu)}) = \mumean \tau_0(\mu) + \frac{\mu}{k}.
\]
Hence, the initial cash level needs to be at least this high to survive until $\mu \geq 0$, i.e.,
\[
	V(x, \mu) = x, \quad \text{if} \quad
	x < - \mumean \tau_0(\mu) - \frac{\mu}{k} =: x_b(\mu).
\]
At an initial cash level $x$ above $x_b(\mu)$, we identify two possible strategies: Either pay out dividends of size $x - x_b(\mu)$ and wait for $\process{\mu}$ to reach 0, or perform a liquidation by paying out $x$. Which strategy is optimal depends on the cost of waiting and the value of future cash flows. Hence, for $x \geq x_b(\mu)$,
\[
    V(x, \mu) = \max\{ x, x - x_b(\mu) + e^{-r \tau_0(\mu)}V(0,0)\} = x + \max\left\{ 0, e^{-r \tau_0(\mu)} V(0,0) - x_b(\mu) \right\}.
\]
Since $x_b$ and $\tau_0$ are both decreasing in $\mu$, there exists a $\mu^*$ such that $e^{-r\tau_0(\mu^*)} V(0,0) = x_b(\mu)$, so from the last term we see that if $\mu \leq \mu^*$, it is optimal to liquidate regardless of the cash level. In the model, this corresponds to paying all cash reserves as dividends at time $t=0$, yielding the value $V(x, \mu) = x$.

With $x_b(\mu) = 0$ for $\mu \geq 0$, we have proved the following result:
\begin{theorem}
    There exist thresholds $x_b(\mu)$ and $\mu^*$ such that
    \begin{itemize}
	\item it is optimal to liquidate immediately if $\mu \leq \mu^*$;
	\item it is optimal to liquidate immediately if $x < x_b(\mu)$;
	\item if $x \geq x_b(\mu)$ and $\mu \geq \mu^*$, it is optimal to immediately pay the excess $x - x_b(\mu)$ and thereafter all earnings as they arrive.
    \end{itemize}
\end{theorem}

\section{Model extensions}
\label{sec:extensions}

In this section we present a few model extensions that we study numerically. The numerical method used is the same as in Section \ref{sec:numerics}, but in the case of fixed costs, a slight generalization of the approximation of controls is necessary, due to the non-local behavior, as mentioned below.

\subsection{Equity issuance}
A firm in need of liquidity could see itself issuing equity to outside investors. In the sequel, we assume that this happens whenever desired, but at a cost. We consider two costs: one cost $\lambda_p$ \emph{proportional} to the capital received and one \emph{fixed cost} $\lambda_f$ which is independent of the amount of equity issued. Mathematically, we follow the model in \cite{decamps2011free} and write
\[ \dif{X}_t = \mu_t \dif{t} + \sigma \dif{W}_t - \dif{L}_t + \dif{I}_t, \]
where $I = \process{I}$, just like $L$, is an adapted, increasing, RCLL control process. We allow for the costs to be $\mu$-dependent,\footnote{This reflects the fact that a more profitable company (higher $\mu$) typically has better access to financial markets.} and write $\lambda_p(\mu_t)$ and $\lambda_f(\mu_t)$. For emphasis, we keep this dependence explicit.

The figures presented in this section are generated with the Ornstein--Uhlenbeck model
\[  \dif{\mu}_t = k(\mumean - \mu) \dif{t} + \musigma \dif{\muW}_t \]
for $k = 0.5$, $\mumean = 0.15$, and $\musigma = 0.3$. The other parameter choices are $\sigma = 0.1$, $\rho = 0$, and $r = 0.05$.

\subsubsection{Proportional issuance costs}
If the costs are purely proportional, i.e., $\lambda_f = 0$, the payoff corresponding to any two controls $L$ and $I$ is
\[  J(x, \mu; L, I) = \E \left[ \int_0^{\theta(L,I)} e^{-rt} \left(\dif{L} - (1+\lambda_p(\mu_t))\dif{I}\right)_t \right],  \]
where $\theta(L, I)$ is the first time the process $X$ becomes negative. In this case, the DPE bears great resemblance to that of the original model, since issuance simply has the opposite effect of dividend payments:
\[  \min \{ rV - \mathcal{L}V, \quad V_x - 1, \quad 1 + \lambda_p(\mu) - V_x \} = 0. \]
The interpretation is that the state space consists of three different regions defined by the optimal action: pay dividends, issue equity, or doing neither. Equity is thus issued whenever $V_x(x,\mu) = 1 + \lambda_p(\mu)$. This means that issuance occurs whenever the marginal value is equal to the marginal cost.

Since issuance is costly and can be done at any time, it is optimal to only issue equity at points where ruin would otherwise be reached, i.e., where $x=0$. However, whether to do so at the boundary depends on the current profitability. Indeed, as seen in Figure \ref{fig:issuance}, equity is only issued when the profitability is above a certain level, below which we still see the band structure of the original problem.
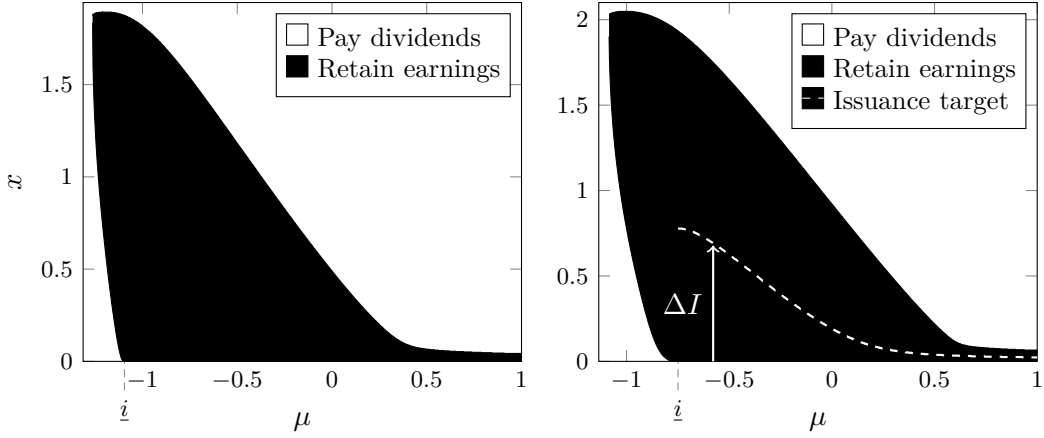
\begin{figure}
    \centering
	\pgfplotsset{every tick label/.append style={font=\footnotesize}}
	\begin{tikzpicture}
	    \begin{axis}[name=axis, ylabel=$x$, xlabel=$\mu$, xlabel shift=-1em, width=0.5\textwidth, enlargelimits=false, enlarge x limits={abs value=0.05,lower}, enlarge y limits={abs value=0.05,upper}, no markers, legend pos=south east, legend cell align=left,
		extra x tick style={
			tick align=outside,
			tick pos=left,
			major tick length=1.0em,
			major tick style={densely dashed}
		    },
		extra x ticks={-1.0953},
		extra x tick labels={$\underline{i}$}]
		\addplot[name path=divLower] table [x=mu, y=divLower, col sep=comma] {Figures/proportional.csv};
		\addplot[name path=divUpper] table [x=mu, y=divUpper, col sep=comma] {Figures/proportional.csv};
		\addplot[fill=black] fill between[of=divLower and divUpper];
	    \end{axis}
	    \node[draw=black, anchor=north east, below left=2mm] at (axis.north east) {%
		    \small
		    \begin{tabular}{@{}r@{ }l@{}}
			\tikz{\draw[black] (0,0) rectangle (0.8em, 0.8em);}&Pay dividends \\
			    \tikz{\filldraw[draw=black, fill=black] (0,0) rectangle (0.8em, 0.8em);}&Retain earnings
			\end{tabular}};
	\end{tikzpicture}
	\begin{tikzpicture}
	    \begin{axis}[name=axis, xlabel=$\mu$, xlabel shift=-1em, width=0.5\textwidth, enlargelimits=false, enlarge x limits={abs value=0.05,lower}, enlarge y limits={abs value=0.05,upper}, no markers, legend pos=south east, legend cell align=left,
		extra x tick style={
			tick align=outside,
			tick pos=left,
			major tick length=1.0em,
			major tick style={densely dashed}
		    },
		extra x ticks={-0.7493},
		extra x tick labels={$\underline{i}$}]
		\addplot[name path=divLower] table [x=mu, y=divLower, col sep=comma] {Figures/fixed.csv};
		\addplot[name path=divUpper] table [x=mu, y=divUpper, col sep=comma] {Figures/fixed.csv};
		\addplot[fill=black] fill between[of=divLower and divUpper];
		\addplot[name path=issuanceTarget, white, dashed, thick] table [x=mu, y=issuanceTarget, col sep=comma] {Figures/fixed.csv} node[pos=0.1, inner sep=0pt] (targetArrow) {};
		\draw[->, white, thick] (targetArrow|-0,0) -- (targetArrow) node[midway, left] {$\Delta I$};
	    \end{axis}
	    \node[draw=black, anchor=north east, below left=2mm] at (axis.north east) {%
		    \small
	    \begin{tabular}{@{}r@{ }l@{}}
		\tikz{\draw[black] (0,0) rectangle (0.8em, 0.8em);}&Pay dividends \\
		\tikz{\filldraw[draw=black, fill=black] (0,0) rectangle (0.8em, 0.8em);}&Retain earnings \\
		\tikz{\filldraw[draw=black, fill=black] (0,0) rectangle (0.8em, 0.8em);\draw[dashed, white] (0em, 0.4em) -- (0.8em, 0.4em);} & Issuance target
	    \end{tabular}};
	\end{tikzpicture}
    \caption{\label{fig:issuance}In the left figure, proportional issuance costs start at 34 \% for low $\mu$ and decaying to 25 \%. Equity is issued according to local a time at the boundary $x=0$ where otherwise ruin would occur, and no equity is issued below $\underline{i} = -1.0953$. In the right figure fixed costs are present, starting at $0.14$ decaying to $0.06$ as a function of $\mu$. Again, equity is only issued at the boundary, and no equity is issued below $\underline{i} = -0.7493$. The dashed white line indicates the cash reserve level after issuance, i.e., how much equity was issued.}
\end{figure}

\subsubsection{Fixed issuance costs}
On the other hand, if the fixed cost is nonzero, we can assume, without loss of generality, that
\[  I_t = \sum_{k = 1}^\infty i_k \1_{\{t \geq \tau_k\}}, \]
for some strictly increasing sequence of stopping times $\tau_k$ and positive $\mathcal{F}_{\tau_k}$-measurable random variables $i_k$.
The stopping times are interpreted as issuance dates, and the random variables as the issued equity.
If $I_t$ could not be written in this form, it would imply infinite issuance frequency, which would, because of the fixed cost, come at infinite cost. This form is therefore a natural restriction, and the corresponding payoff functional is given by
\[  J(x, \mu; L, I) = \E \left[ \int_0^{\theta(L,I)} e^{-rt} \dif{L}_t - \sum_{k=1}^\infty e^{-r \tau_k} (\lambda_f(\mu_t) + \lambda_p(\mu_t) i_k) \1_{\{\tau_k < \theta(L, I)\}} \right]. \]
The value function given by this problem is then the solution to the following nonlocal DPE:
\[  \min \left\{ rV - \mathcal{L}V, \quad V_x -1, \quad V(x, \mu) - \sup_{i \geq 0} \left(V(x+i, \mu) - \lambda_p i - \lambda_f\right) \right\} = 0. \]
The last conditions states that the value at any given point is at least equal to the value in any point after issuance less the issuance costs.

Just like for proportional costs, issuance optimally only occurs at the boundary. However, with fixed costs, the amount of equity issued is now larger in order to avoid incurring another fixed cost soon in the future. The magnitude is presented as the issuance target in Figure \ref{fig:issuance}. Note that, the numerical method employed in the fixed cost case can be interpreted as the issuance structure in \cite{hugonnier2015capital} where the arrival of investors might not coincide with the desire for equity issuance. The approximation parameter $K$ from the singular case will here describe the arrival rate of investors, and the value $V$ is the limit when the investor arrival rate tends to infinity. This same approximation scheme is presented for another problem in \cite{altarovici2017optimal} and turns out to be very accurate.

By letting the fixed cost grow sufficiently fast in $-\mu$, one can obtain substantially different issuance policies. As shown in Figure \ref{fig:wave}, such structure can have a wave-like shape, not dissimilar to the shape of the continuity region. This seems to indicate two factors at play: Either one issues equity as a last resort at $x=0$, or at an earlier time in fear of higher issuance costs in the future. Moreover, in this regime, the target points no longer constitute a continuous line, but instead has a jump discontinuity to the right of the gray region.

\begin{figure}[h]
    \centering
\begin{tikzpicture}
	\pgfplotsset{every axis/.append style={axis on top,
			width=\textwidth,
			enlargelimits=false,
			scaled ticks=false,
			tick label style={/pgf/number format/fixed},
		}
	}
	\begin{axis}[name=axisborder, xlabel=$\mu$, ylabel=$x$]
		\addplot graphics[xmin=-1.0, ymin=0, xmax=1.0, ymax=2.15] {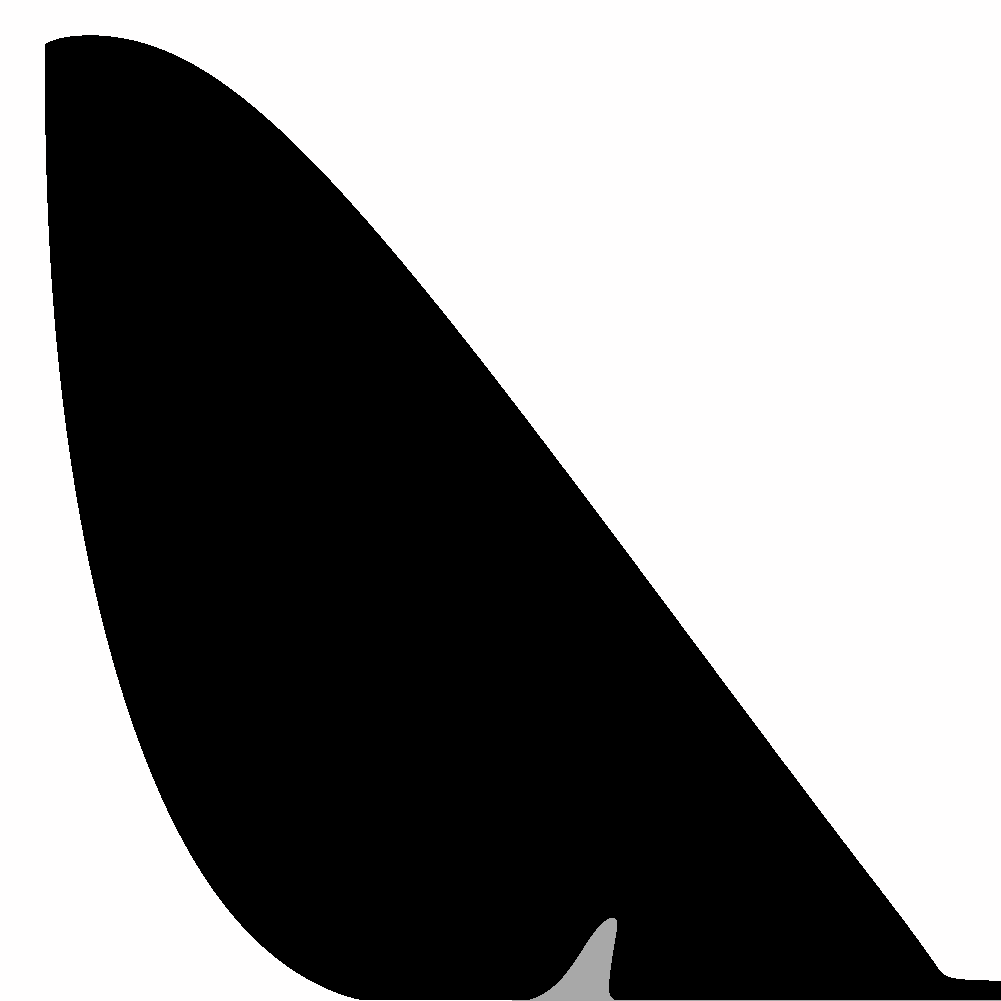};
		\addplot[unbounded coords=jump,name path=issuanceTarget, white, dashed, thick] table [x=mu, y=issuanceTarget, col sep=comma] {Figures/wavetarget.csv};
	\end{axis}
		\node[draw=black, anchor=north east, below left=2mm] at (axisborder.north east) {%
			\begin{tabular}{@{}r@{ }l@{}}
				\tikz{\draw[black] (0,0) rectangle (3mm, 3mm);}&Pay dividends \\
				\tikz{\filldraw[draw=black, fill=black] (0,0) rectangle (3mm, 3mm);}&Retain earnings \\
				\tikz{\filldraw[draw=black, fill=lightgray] (0,0) rectangle (3mm, 3mm);}&Issue equity \\
				\tikz{\filldraw[draw=black, fill=black] (0,0) rectangle (0.8em, 0.8em);\draw[dashed, white] (0em, 0.4em) -- (0.8em, 0.4em);} & Issuance target
			\end{tabular}};
\end{tikzpicture}
\caption{\label{fig:wave}When issuance costs are sufficiently high for negative $\mu$, it is optimal to issue equity away from the boundary $x=0$.}
\end{figure}
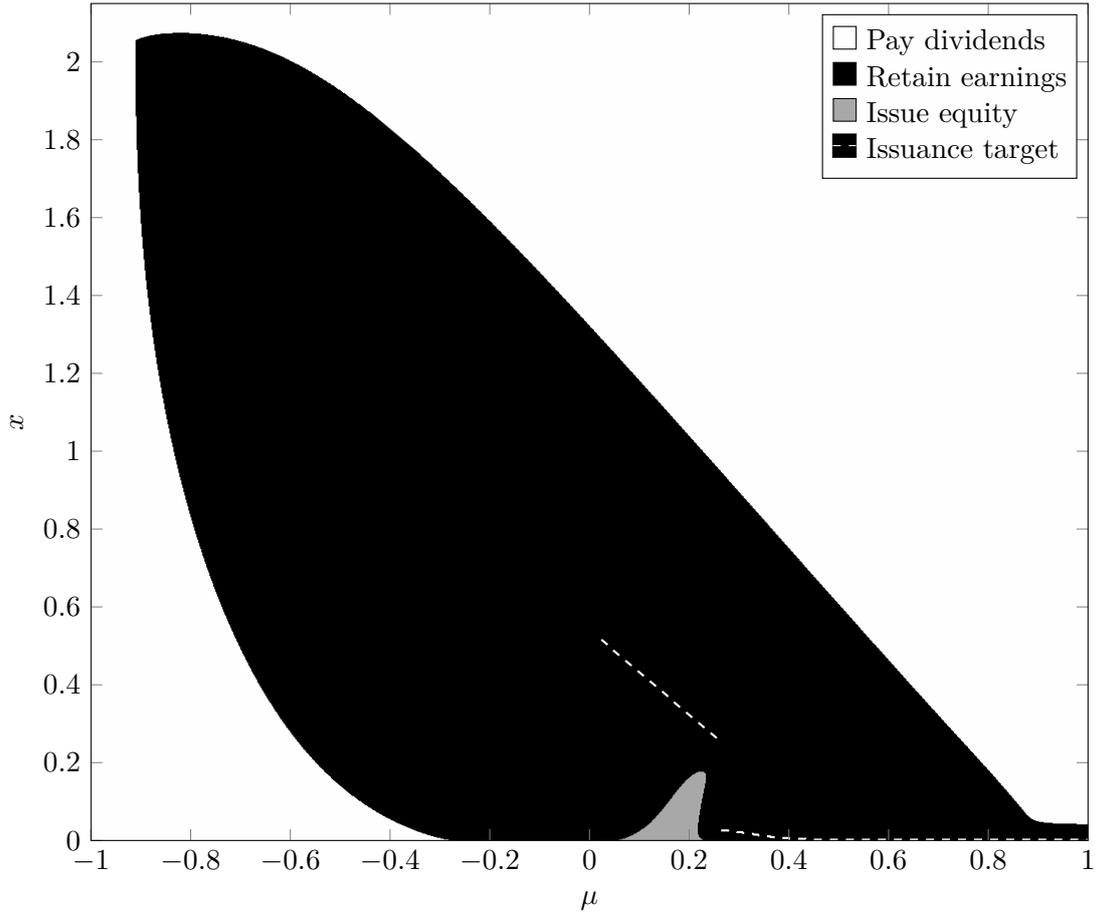

\subsection{Credit lines}
If the firm is sufficiently profitable ($\mu$ large enough), it could be granted a credit line by a bank, whereby the cash reserves $X = \process{X}$ are allowed to be negative up to a certain threshold $\underline{x}(\mu)$, below which bankruptcy occurs. Suppose the interest on this credit line is $\rho_- \geq 0$ and define
\[  \rho(x) = \begin{cases}
	0,\quad & x \geq 0 \\
	\rho_-, \quad & x < 0
    \end{cases}. \]
The cash balance (reserves) process can then be written as
\[  \dif{X}_t = (\rho(X_t)X_t + \mu_t) \dif{t} + \sigma \dif{W}_t - \dif{L}_t. \]
Although the credit line has the effect of shifting the dividend region downwards, closer to the new bankruptcy boundary, as seen in Figure \ref{fig:creditline}, the effect of the interest rates $\rho$ seems to be minimal. Indeed, with interest rates of order 1 \% and optimal cash levels of order 1, the effecting increase in drift is then also of order 1 \%. This effect is therefore generally small in comparison to the magnitude of the profitability.
\begin{figure}[h]
    \centering
\begin{tikzpicture}
	\pgfplotsset{every axis/.append style={axis on top,
			width=0.6\textwidth,
			enlargelimits=false,
			scaled ticks=false,
			tick label style={/pgf/number format/fixed},
		}
	}
	\begin{axis}[name=axisborder, xlabel=$\mu$, ylabel=$x$]
		\addplot graphics[xmin=-1.1, ymin=-0.9499999998597367, xmax=1.2219, ymax=2.9] {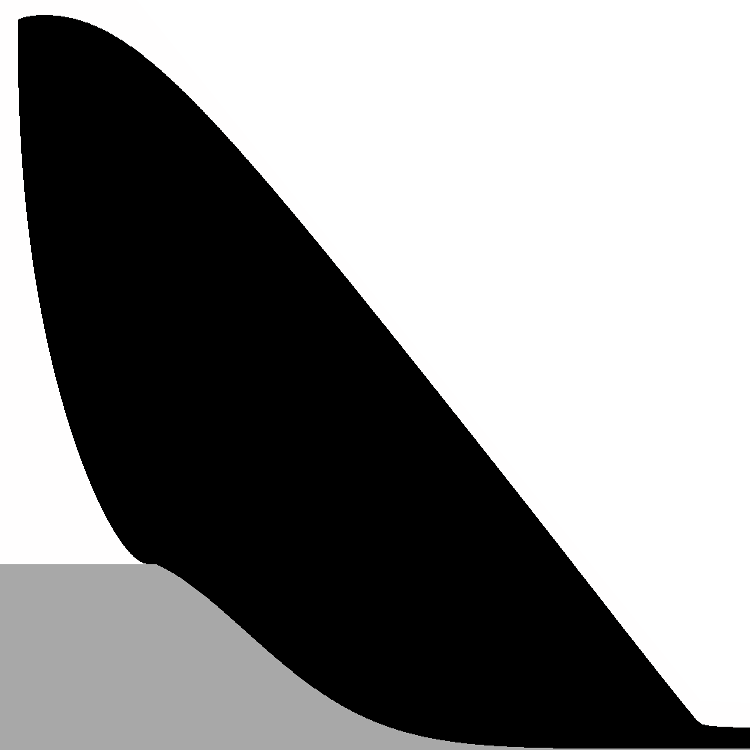};
	\end{axis}
		\node[draw=black, anchor=north east, below left=2mm] at (axisborder.north east) {%
			\begin{tabular}{@{}r@{ }l@{}}
				\tikz{\draw[black] (0,0) rectangle (3mm, 3mm);}&Pay dividends \\
				\tikz{\filldraw[draw=black, fill=black] (0,0) rectangle (3mm, 3mm);}&No dividends \\
				\tikz{\filldraw[draw=black, fill=lightgray] (0,0) rectangle (3mm, 3mm);}&Ruin \\
			\end{tabular}};
\end{tikzpicture}
\caption{\label{fig:creditline}The credit line shifts the upper boundary downwards, but the main characteristics of the model are otherwise retained.}
\end{figure}

\section{Concluding comments}
\label{sec:conclusion}
In this paper, we have studied an optimal dividend problem where the profitability of the firm is stochastic.
In this context, we have given a comparison result for general dynamics as well as outlined a numerical algorithm to compute the solution.
The convergence of the numerical method is a consequence of the comparison result.

The numerical solution indicates that the problem exhibits both barrier and band structure, depending on the current profitability.
The band structure is especially interesting, since it is unusual for Brownian dynamics.
We also prove that when the profitability falls below a certain threshold, liquidation becomes optimal.
Although mean reversion indicates that the profitability of the company will improve over time, that is not enough;
when the initial profitability is too low, the expected cumulative losses are too large to be covered by the future income, so immediate closure is optimal.

The flexibility of the numerical method from Section \ref{sec:numerics} allows us to numerically study a number of extensions to our original model.
The extended models allow for issuance with proportional and/or fixed costs as well as the possibility of credit lines.
Finally, one extension provides a way to concavify the problem (in the $x$-dimension) by giving the firm the opportunity of entering fair bets.

\section{Proofs}
\label{sec:technical_proofs}
This section is dedicated to the proofs of previous sections.
We begin with a result that is needed in multiple proofs. It is proven under slightly stronger assumptions, which turn out to be satisfied without loss of generality where the lemma is needed.
\begin{lemma}
    \label{lem:maxIntegrability}
    If, in addition to Assumption \ref{ass:regularity}, $\inf \muspace > -\infty$ or $\mu \in \bigO(\mudrift(\mu))$ as $\mu \to -\infty$, there exists a sublinearly growing function $H$ and a constant $C$ such that, for any stopping time $\tau$,
    \[  \E \left[ \max_{0\leq t \leq \tau} |\mu^0_t| \right] \leq C \E [H(\tau)]. \]
\end{lemma}

\begin{proof}
    We will use the result by Peskir \cite{peskir2001} to obtain a function $H$ which is sublinearly growing.  For some $c \in \interior{\muspace}$, let
    \[  S'(\mu) = \exp\left(- 2 \int_c^\mu \frac{\mudrift(\nu)}{\musigma(\nu)^2} \dif{\nu}\right) \]
    and
    \[  m(\dif{\nu}) = \frac{2 \dif{\nu}}{S'(\nu) \musigma(\nu)^2}. \]
    Finally, define
    \[  F(\mu) = \int_c^\mu m((c, \nu]) S'(\nu) \dif{\nu}. \]
    Note that by the assumptions of the statement, $\mudrift$ and $\musigma$ behave analogously for large positive and negative $\mu$, so, without loss of generality, we may consider only positive values of $\mu$. In particular, for large $\mu$, $S'(\mu)$ grows as $\exp(a \mu^\gamma)$ for some $a > 0$ and $\gamma \geq 1$, and we are done if we can verify the following condition:
    \[  \sup_{\mu > c} \left( \frac{F(\mu)}{\mu} \int_\mu^\infty \frac{\dif{\nu}}{F(\nu)} \right) < \infty. \]
    All involved functions are continuous, so we are done if it has a finite limit (or is negative) as $\mu \to \infty$.
    L'Hôpital's rule yields the fraction
    \[  \frac{\od{}{\mu} \int_\mu^\infty \frac{\dif{\nu}}{F(\nu)}}{\od{}{\mu} \frac{\mu}{F}} = \frac{F(\mu)}{\mu F'(\mu) - F(\mu)}. \]
    If the denominator were bounded from above, Grönwall's inequality would imply linear growth of $F$, which contradicts the growth of $S'$. Hence, the expression is either negative (and we are done), or we may use l'Hôpital's rule again:
    \[  \frac{F'(\mu)}{\mu F''(\mu)} = \frac{S'(\mu) m((c, \mu])}{\frac{2 \mu}{\musigma(\mu)^2} + \mu S''(\mu) m((c, \mu])} \xrightarrow{\mu \to \infty} 0, \]
    since $S'(\mu)/\mu S''(\mu) = \musigma(\mu)^2 / (-2 \mudrift(\mu) \mu) \to 0$. Thus, with $H = F^{-1}$, \cite{peskir2001} allows us to conclude that
    \begin{equation}
	\label{eqn:maxMuBound}
	\E \left[ \max_{0 \leq s \leq \tau} |\mu_s| \right] \leq C \E [ H(\tau) ],
    \end{equation}
    for some constant $C$ and any stopping time $\tau$. In particular, for $\tau = t$, the expression is finite and sublinearly growing in $t$.
\end{proof}

\subsection{Liquidation threshold}
\label{sec:proofarbsign}
Before we present the proof of Theorem \ref{thm:liquidationThreshold}, we present an auxiliary problem whose properties are especially useful to prove the existence of a liquidation threshold.

Consider the case where $L$ is not restricted to be nondecreasing. This means that shareholders may inject new cash into the firm at no cost. In that case, cash reserves are useless and $x$ must be distributed right away:
\[  \rovalue (x, \mu) = x + \rovalue(\mu), \]
where the auxiliary function $\rovalue$ is given by
\[  \rovalue(\mu) = \sup_\tau \E \left[ \int_0^\tau e^{-rt} \mu_t \dif{t} \right]. \]
The liquidation time $\tau$ is chosen freely by the shareholders of the firm. This is a real option problem (see for example \cite{dixit1994investment}). Intuitively, the owners of the firm exert the liquidation option when the profitability $\mu$ falls below a (negative) threshold $\mu^*$, provided $\process{\mu}$ can reach such a point. In particular, $\rovalue$ and $\mu^*$ satisfies the following boundary value problem:
\[ r \rovalue(\mu) - \mudrift(\mu) \rovalue'(\mu) - \frac{\musigma(\mu)^2}{2} \rovalue''(\mu) = \mu, \]
subject to
\[  \rovalue(\mu^*) = \rovalue'(\mu^*) = 0. \]

The second part of this theorem characterizes an optimal liquidation time. This proves useful in the original problem since $V(x, \mu) \leq \rovalue(x, \mu)$, thus showing that $V(x, \mu) \equiv x$ for levels of $\mu$ below some threshold.

\begin{lemma}
    \label{thm:realoption}
    When dividends can be of arbitrary sign, the optimal policy for shareholders is to immediately distribute the initial cash reserves at $t=0$, and to maintain them at zero forever by choosing $\dif{L}_t = \mu_t \dif{t} + \xsigma \dif{W}_t$. If $\muspace$ has no lower bound, there exists a $\mu^*$ such that the firm is liquidated when profitability falls below the threshold $\mu^*$: $  \tau = \inf \{ t > 0 : \mu_t \leq \mu^* \} $ is a maximizer.
\end{lemma}

\begin{proof}
	Suppose $L$ be any strategy for which $X^L_t \geq 0$ until some (liquidation) time $\tau$. Then define
	\[  L'_t = L_t + X^L_t.  \]
	Since $X^L_t$ is nonnegative until $\tau$, it is clear that $X^{L'}_t = 0$ and that $L'_t \geq L_t$ for $t \leq \tau$. Hence, $L'$ is admissible whenever $L$ is, and it also produces a higher payoff.

    It remains to prove the existence of $\mu^*$. If $\muspace$ has no lower bound, but $\mudrift$ is not growing linearly as $\mu \to -\infty$, consider instead of $\process{\mu}$ another process with the same $\musigma$, but which also fulfills this growth condition. The corresponding value function dominates our original one, so it is enough to prove it in this case.

    Setting $\dif{L}_t = \mu_t \dif{t} + \xsigma \dif{W}_t$ until a stopping time $\tau$ yields
    \[  J(x, \mu; L) = x + \E \int_0^{\tau} e^{-rt} \mu_t \dif{t} + \E \int_0^{\tau} e^{-rt} \xsigma \dif{W}_t. \]
    Since the last term is zero, the value function is obtained by maximizing over $\tau$:
    \[  V(x, \mu) = x + \sup_\tau \E \int_0^{\tau} e^{-rt} \mu_t \dif{t} = x + \rovalue(\mu). \]

    We now try to find a point in which $\rovalue$ is $0$.
    Consider the equation
    \begin{equation}
	\label{eqn:realoptionDPE}
	\min \left \{ - \mu + r \phi - \mudrift (\mu) \phi' - \frac{1}{2} \musigma(\mu)^2 \phi'', \phi \right\} = 0,
    \end{equation}
    and suppose it has a solution which never touches $0$, i.e, $\phi > 0$. Then,
    \[  \phi(\mu) = \E \int_0^\infty e^{-rt} \mu_t \dif{t} = \int_0^\infty e^{-rt} \E [\mu_t] \dif{t}. \]
	Using Itô's formula, the bounds on $\mudrift$ and $\musigma$, as well as \eqref{eqn:maxMuBound}, for $\mu < 0$, we obtain
	\[ \E [\mu_t] \leq \mu + \int_0^t C_1 \left( 1 + \E\left[ \sup_{0 \leq s \leq t} |\mu_s^0|\right]\right) \dif{s} \leq \mu + t C_2 (1 + H(t)), \]
    for some constants $C_1$ and $C_2$. Hence,
	\[  \phi(\mu) \leq \int_0^\infty e^{-rt} \left(\mu + C_2 H(t)\right) \dif{t} \leq \frac{\mu}{r} + C_3, \]
    for another constant $C_3$.
    Thus, $\phi(\mu) \to - \infty$ as $\mu \to -\infty$, which contradicts that $\phi \geq 0$. We conclude that a solution $\phi$ must indeed touch $0$.

	Finally, we are done if $\rovalue$ satisfies the dynamic programming equation \eqref{eqn:realoptionDPE}. By Lemma \ref{lem:maxIntegrability} and \cite{kobylanski2012optimal}, the optimal stopping time is the hitting time of $A_0 = \{\mu : \rovalue(\mu) = 0 \} \neq 0$. Hence, the function is smooth everywhere, except possibly at $\mu^* := \sup A_0$. However, since $\musigma$ never vanishes, an argument analogous to in the proof of Theorem \ref{thm:contatzero} yields continuity also at $\mu^*$, from which \eqref{eqn:realoptionDPE} can be derived.
\end{proof}

\liquidationThreshold*
\begin{proof}
    Until the time of ruin $\theta(L)$, $L_t \leq x + \int_0^t \mu_s \dif{s} + \int_0^t \xsigma \dif{W}_s$. Hence, by stochastic integration by parts,
    \[  V(x,\mu) = \sup_L \E \int_0^{\theta(L)} e^{-rt} \dif{L}_t \leq x + \sup_L \E \int_0^{\theta(L)} e^{-rt} \mu_t \dif{t} + \E \int_0^{\theta(L)} e^{-rt} \xsigma \dif{W}_t. \]
    First observe that the last term is equal to $0$. Then, since the second term is smaller than or equal to $\rovalue(\mu)$, the result is a direct consequence of Theorem \ref{thm:realoption}.
\end{proof}

\subsection{Continuity}

When proving the continuity of the value function, we need a weak form of a dynamic programming \emph{inequality}. More precisely, for any control $L$ and stopping time $\tau$ with values between $0$ and $\theta(L)$,
\begin{equation}
	\label{eqn:easyWDPP}
	E \left[ \int_\tau^{\theta(L)} e^{-rt} \dif{L}_t \middle| \mathcal{F}_{\tau} \right] \leq e^{-r\tau} \bar{V}(X_\tau, \mu_\tau), \quad P\text{-a.s.},
\end{equation}
where $\bar{V}$ denotes the upper-semicontinuous envelope of $V$.
The reason we need to use $\bar{V}$ and not $V$ is that we do not know a priori whether $V$ is measurable. This inequality very much related to the weak dynamic programming principle \cite{bouchard2011} which also establishes a similar inequality in the other direction. However, \eqref{eqn:easyWDPP} is more primitive than the inequality in the other direction.

These measurability issues are arguably the most notable obstacles in establishing the dynamic programming principle. However, for continuous value functions, proofs of the dynamic programming principle are well known \cite{fleming2006controlled}. For the general case, we once again refer to \cite{karoui2013capacitiesI,karoui2013capacitiesII}.

In this section we establish the continuity of the value function, from which the dynamic programming principle then follows.

\begin{theorem}
    \label{thm:contatzero}
	The value function is continuous at $x = 0$.
\end{theorem}
\begin{proof}
	Let $\{(x^n, \mu^n)\}_{n \geq 1}$ be a sequence converging to $(0, \mu^\infty)$. Without loss of generality, assume $x^n > 0$ converges monotonically to $0$. Since $V \equiv 0$ at $x=0$, it is sufficient to consider monotonically decreasing sequences in $\mu$, by monotonicity in $\mu$. For simplicity, also assume that $x^1 < 1$ and $|\mu^1 - \mu^\infty| < 1$.

	Let $\tau$ be any strictly positive, bounded stopping time such that for $t \leq \tau$, $\mu^{\mu^1}_t \leq |\mu^\infty| + 1$ and $X^{(x^1, \mu^1), 0}_t \leq 1$, $P$-a.s for the uncontrolled process corresponding to $L = 0$. Hence, we also have $X^{(x^n, \mu^n), L} \leq X^{x^1, \mu^1, 0} \leq 1$ for all $n \in \mathbb{N}$ and dividend processes $L$.

	Starting with the definition of the value function, one obtains
	\begin{align*}
		V(x^n, \mu^n) &= \sup_L E\left[ \int_0^{\tau \wedge \theta_n(L)} e^{-r t} \dif{L}_t + 1_{\{\tau < \theta_n(L)\}} \int_{\tau \wedge \theta_n(L)}^{\theta_n(L)} e^{-rt} \dif{L}_t \right] \\
					&\leq \sup_L E\left[ \int_0^{\tau \wedge \theta_n(L)} e^{-r t} \dif{L}_t + 1_{\{\tau < \theta_n(L)\}} \int_{0}^{\theta_n(L)} e^{-rt} \dif{L}_t \right] \\
					&\leq \sup_L E \left[ x^n + (|\mu^\infty| + 1) (\tau \wedge \theta_n(L)) + V(1, |\mu^\infty| + 1) 1_{\{\tau < \theta_n(L)\}} \right] \\
			& \leq x^n + (|\mu^\infty| + 1) E[\tau \wedge \theta_n(0)] + V(1, |\mu^\infty| + 1) P[\tau < \theta_n(0)].
	\end{align*}

	Now make the observation that,
	\[
		\{ \tau < \theta_{n+1}(0) \} \subseteq \{ \tau < \theta_n(0) \}.
	\]
	Since $\sigma > 0$, $\theta_n(0) \rightarrow 0$ $P$-a.s., and therefore
	\[
		\lim_{n \rightarrow \infty} P[ \tau < \theta_n(0) ] = P\left(\lim_{n \rightarrow \infty} \{\tau < \theta_n(0) \} \right) = P[ \tau \leq 0] = 0.
	\]
	By letting $n\rightarrow \infty$, we obtain $\lim_n V(x^n, \mu^n) \leq (|\mu^\infty| + 1) E[\tau]$, but since $\tau$ can be chosen arbitrarily small, we conclude that
	\[
		\lim_{n \rightarrow \infty} V(x^n, \mu^n) \leq 0.
	\]
	Since $V$ is non-negative and zero where $x=0$,
	\[
		\lim_{n \rightarrow \infty} V(x^n, \mu^n) = 0 = V(0, \mu^\infty).
	\]
	\end{proof}

%
%

\begin{lemma}
    \label{lem:dominatedPayoff}
	For each starting point $(x,\mu)$, the payoffs $\int_0^{\theta(L)} e^{-\gamma t} \dif{L}_t$ for strategies $L$ are uniformly bounded in $\|\|_{L^1(\Omega)}$ for any $\gamma > 0$.
\end{lemma}
\begin{proof}
    Without loss of generality, we may assume that $\mu \in \bigO(\mudrift(\mu))$ as $\mu \to -\infty$, since this yields a larger or equally large process $\mu_t$, and therefore also $\int_0^{\theta(L)} e^{-\gamma t} \dif{L}_t$.
    Integration by parts yields
    \begin{align*}
	\int_0^{\theta(L)} e^{-\gamma t} \dif{L}_t &\leq x + \int_0^{\theta(L)} e^{-\gamma t} \mu_t \dif{t} + \int_0^{\theta(L)} e^{-\gamma t} \sigma(\mu_t) \dif{W}_t \\
	&\leq x + \int_0^{\infty} e^{-\gamma t} |\mu_t| \dif{t} + \sup_T \int_0^T e^{-\gamma t} \sigma(\mu_t) \dif{W}_t,
    \end{align*}
    where the $\sup$ is taken over stopping times $T$. This provides the $L$-independent bound if it has finite expectation.

    The expectation of the first integral is finite, since, by Lemma \ref{lem:maxIntegrability},
    \[  E\left[ \int_0^\infty e^{-\gamma t} |\mu_t| \dif{t} \right] \leq \int_0^\infty e^{-\gamma t} H(t) \dif{t} < \infty. \]
    Similarly, by Doob's inequality, It\^o isometry and Lemma \ref{lem:maxIntegrability}, we obtain, for some $C$,
    \[  E \left[ \left( \sup_T \int_0^T e^{-\gamma t} \sigma(\mu_t) \dif{W}_t \right)^2 \right] \leq 2 \int_0^\infty e^{-2\gamma t} C\left(1 + H(t) \right) \dif{t} < \infty. \]
\end{proof}

\continuity*
	\begin{proof}
	    Denote by $\theta^n(L)$ the bankruptcy time of starting in $(x^n, \mu^n)$ and using the dividend policy $L$. Denote by $L^n$ $\varepsilon$-optimal stratgies starting in $(x^n, \mu^n)$. Similarly, let $X^n_t$ be the process associated with the starting point $(x^n, \mu^n)$ and dividend process $L^n$.
	
	First consider a sequence $(x^n, \mu^n)$ which is non-decreasing in both coordinates and converges to $(x^\infty, \mu^\infty)$. Then
	\begin{equation}
		\label{eqn:contvalbound}
		\begin{aligned}
		    V(x^\infty, \mu^\infty) - \varepsilon &\leq E\left[ \int_0^{\theta^n(L^\infty)} e^{-rt} \dif{L}^\infty_t + \int_{\theta^n(L^\infty)}^{\theta^\infty(L^\infty)} e^{-rt} \dif{L}^\infty_t \right] \\
			& \leq V(x^n, \mu^n) + E \left[ \int_{\theta^n(L^\infty)}^{\theta^\infty(L^\infty)} e^{-rt} \dif{L}^\infty_t \right].
		\end{aligned}
	\end{equation}
	Therefore, if we can show that the second term tends to zero as $n$ tends to infinity, we are done.

	By Lebesgue's dominated convergence theorem (see Lemma \ref{lem:dominatedPayoff}),
	\begin{align}
		\lim_{n\rightarrow\infty} E \left[ \int_{\theta^n(L^\infty)}^{\theta^\infty(L^\infty)} e^{-rt} \dif{L}^\infty_t \right]
		&= E \left[ \lim_{n \to \infty} \int_{\theta^n(L^\infty)}^{\theta^\infty(L^\infty)} e^{-rt} \dif{L}^\infty_t \right].
	\end{align}

	We will prove that the limit inside the expectation is 0 on the following set:
	\begin{align*}
	    \Omega' &= \left\{ \sup_L \int_0^{\theta(L)} e^{-rt} \dif{L}_t < \infty \right\} \\
		 &\quad \cap \left\{ \E \left[ \int_{\theta^n(L^\infty)}^{\theta^\infty(L^\infty)} e^{-rt} \dif{L}^\infty_t \middle| \mathcal{F}_{\theta^n(L^\infty)} \right] \leq e^{-r\theta^n(L^\infty)} \bar{V}\left(X^\infty_{\theta^n(L^\infty)}, \mu^\infty_{\theta^n(L^\infty)}\right), \forall n \in \mathbb{N} \right\},
	\end{align*}
	where $\bar{V}$ denotes the upper semicontinuous envelope of $V$.
	Note that by Lemma \ref{lem:maxIntegrability} and \eqref{eqn:easyWDPP}, $P(\Omega') = 1$. For any $\omega \in \Omega'$, consider the following two cases:

	First, consider any divergent subsequence $\theta^k = \theta^{n(k)}(L^\infty)(\omega)$.
	Then,
	\[  \left( \int_{\theta^k}^{\theta^\infty(L^\infty)} e^{-rt} \dif{L}^\infty_t \right) (\omega) \leq e^{-r \theta^k/2} \left( \int_{\theta^k}^{\theta^\infty(L^\infty)} e^{-rt/2} \dif{L}^\infty_t \right) (\omega) \xrightarrow{k \to \infty} 0, \]
	because $\omega \in \Omega'$ ensures that the integral factor is bounded, and the exponential factor converges to 0.

	On the other hand, let $\theta^k = \theta^{n(k)}(L^\infty)(\omega)$ be a bounded subsequence.
	We then wish to use the continuity of $V$, and therefore also of $\bar{V}$, at 0 to argue that
	\[
	    \lim_{k\rightarrow\infty} \bar{V}\left(X^\infty_{\theta^k}(\omega), \mu^\infty_{\theta^k}(\omega) \right) = 0.
	\]
	Since $\sup_k \theta^k(\omega) < \infty$, $X^\infty_{\theta^k} = X^\infty_{\theta^k} - X^k_{\theta^k} = x^\infty - x^k + \int_0^{\theta^k} \mu^\infty_t - \mu^k_t \dif{t} \xrightarrow{k\to\infty} 0$, because of continuity with respect to initial points. Therefore, since $V$ is increasing, it is sufficient to find a bound to $\mu^\infty_{\theta^k}(\omega)$.
	
	Begin by considering the process $M^\infty_t = \sup_{0 \leq s \leq t} \mu^\infty_s$. Then, since $\theta^k$ is a bounded sequence and $M^\infty_t$ is continuous, $M^\infty_{\theta^k}(\omega)$ is bounded by some constant $C$. Therefore, by Theorem \ref{thm:contatzero},
	\[ \lim_{k\rightarrow\infty} \bar{V}\left(X^\infty_{\theta^k}(\omega), \mu^\infty_{\theta^k}(\omega) \right) \leq \lim_{k\rightarrow\infty} \bar{V}\left(X^\infty_{\theta^k}(\omega), C\right) = \bar{V}(0, C) = 0.  \]

	By the preceding arguments, it holds that for any bounded subsequence and any subsequence converging to $\infty$,
	\[\left( \int_{\theta^n(L^\infty)}^{\theta^\infty(L^\infty)} e^{-rt} \dif{L}^\infty_t \right) (\omega)\]
	converges to 0 for every $\omega \in \Omega'$.
	As a consequence of this, the whole sequence converges to 0, for every $\omega \in \Omega'$, i.e., $P$-a.s.
	Returning to \eqref{eqn:contvalbound}, this yields
	\[
		V(x^\infty, \mu^\infty) - \varepsilon \leq \lim_{n\rightarrow\infty} V(x^n, \mu^n) \leq V(x^\infty, \mu^\infty),
	\]
	by monotonicity. Since this holds for any choice of $\varepsilon>0$, equality is obtained.

	Now let $(x^n, \mu^n)$ be non-increasing in both coordinates and converge to $(x^\infty, \mu^\infty)$. Then, in the same manner as above,
	\[
		\lim_{n\rightarrow\infty} V(x^n, \mu^n) - \varepsilon \leq V(x^\infty, \mu^\infty) + E\left[ \lim_{n\rightarrow\infty} \int_{\theta^\infty(L^n)}^{\theta^n(L^n)} e^{-rt} \dif{L}^n_t \right],
	\]
	and by analogous arguments,
	\[
		\lim_{n\rightarrow\infty} V(x^n, \mu^n) = V(x^\infty, \mu^\infty).
	\]

	As a final step, consider an arbitrary convergent sequence $(x^n, \mu^n)$. By monotonicity,
	\[
		V\left(\inf_{m \geq n} x_m, \inf_{m \geq n} \mu_m\right) \leq V(x^n, \mu^n) \leq V\left(\sup_{m \geq n} x_m, \sup_{m \geq n} \mu_m\right).
	\]
	Since the sequences $(\inf_{m \geq n} x_m, \inf_{m \geq n} \mu_m)$ and $(\sup_{m \geq n} x_m, \sup_{m \geq n} \mu_m)$ are non-decreasing and non-increasing, respectively, it follows that
	\[
		\lim_{n\rightarrow\infty} V(x^n, \mu^n) = V(x^\infty, \mu^\infty),
	\]
	and $V$ is continuous everywhere.
\end{proof}

\subsection{Comparison principle}

	\begin{lemma}
	\label{lem:transformed_dpe}
		If a function $u$ is a viscosity subsolution to \eqref{eqn:dpe}, then
		\begin{equation} \label{eqn:transformation}
		    \tilde{u}(x, \mu) := e^{-\eta x - \eta g(\mu)} u(x, \mu)
		\end{equation}
		is a viscosity subsolution to
		\begin{equation}
			\label{eqn:transformed_dpe}
			\begin{aligned}
				\min \Big\{ & \Big(r - \eta \mu - \eta g'(\mu) \kappa(\mu) - \eta^2 \Sigma_{11} \\
				& \quad - \eta^2 g'(\mu)^2 \Sigma_{22} - \eta g''(\mu) \Sigma_{22} - 2 \eta^2 g'(\mu) \Sigma_{12} \Big) \tilde{V} \\
				& - (\mu + \eta \Sigma_{11} + 2 \eta g'(\mu) \Sigma_{12}) \tilde{V}_x \\
				& - (\kappa(\mu) + \eta g'(\mu) \Sigma_{22} + 2 \eta \Sigma_{12} ) \tilde{V}_\mu \\
				& - \Tr \Sigma D^2 \tilde{V}, \\
				& \qquad \qquad \eta \tilde{V} + \tilde{V}_x - e^{-\eta x - \eta g(\mu)}  \Big\} = 0, \text{ in } \mathbb{R}_{>0} \times \muspace,
			\end{aligned}
		\end{equation}
		for any given $\eta$ and $g \in C^2(\mathbb{R})$. A corresponding statement is true for supersolutions.
	\end{lemma}

	\begin{proof}
		Let $u$ be a viscosity subsolution to \eqref{eqn:dpe}. Let $\tilde{\varphi} \in C^2$ and $(x_0, \mu_0)$ be a local maximum of $\tilde{u} - \tilde{\varphi}$ where $(\tilde{u} - \tilde{\varphi})(x_0, \mu_0) = 0$. Finally, define
		\[
		    \varphi(x, \mu) := e^{\eta x + \eta g(\mu)} \tilde{\varphi}(x, \mu).
		\]
		Then $(x_0, \mu_0)$ is also a local maximum of $u - \varphi$. Using the viscosity property of $u$ and plugging in $\varphi$ yields the viscosity property for $\tilde{u}$ and the transformed equation \eqref{eqn:transformed_dpe}.
	\end{proof}

	\begin{lemma}
	    The parameter $\eta$ and function $g$ can be chosen in such a way that the coefficient of $\tilde{V}$ in \eqref{eqn:transformed_dpe} is strictly positive.
		\label{lem:monotonicity}
	\end{lemma}
	\begin{proof}
	    Fix any large $y > 0$ and let $g$ be a twice differentible function with
	    \[ \arraycolsep=0pt
		g'(\mu) = \left\{\begin{array}{rll}
			- &\eta_-, &\quad \text{if } \mu < -y, \\
		   &\eta_+, &\quad \text{if } \mu > y, \end{array}\right.
		\]
	    for strictly postive constants $\eta_-$ and $\eta_+$ as well as $\mu \in [-y,y]^c$.\footnote{This can be done by a mollification argument on $g'$.} For any such choice, the coefficient is strictly positive in $[-y, y]$, provided $\eta$ is small enough. Moreover, with our choice of $g$, the condition reduces to
	    \[	r - \eta (\mu + \eta_\pm \kappa(\mu)) - \eta^2 (\Sigma_{11} + \eta_\pm^2 \Sigma_{22} + 2 \eta_\pm \Sigma_{12} ) > 0, \quad \text{in } [-y,y]^c. \]
	    Note that the last two terms both grow at most linearly in $\mu$, by the growth conditions on $\kappa$, $\tilde{\sigma}$, and $\sigma$. Furthermore, since $\kappa$ is negative and linearly growing for large (positive) $\mu$, we can choose $\eta_+$ such that $-(\mu + \eta_+ \kappa(\mu))$ is linearly increasing. Then, for sufficiently small $\eta$, the whole expression is increasing, for large $\mu$.

	    Similarly, if both $\eta_-$ and $\eta$ are chosen sufficiently small, the same is true also for large, negative $\mu$. Hence, for such a choice of $\eta$ and $g$, the coefficient is strictly positive.
	\end{proof}

	\begin{remark}
	    Assumption \ref{ass:regularity} could possibly be generalized by finding strict supersolutions to the equation
	    \[ r - \eta \mu - \eta g'(\mu) \kappa(\mu) - \eta^2 \Sigma_{11} - \eta^2 g'(\mu)^2 \Sigma_{22} - \eta g''(\mu) \Sigma_{22} - 2 \eta^2 g'(\mu) \Sigma_{12} = 0, \]
	    since this is sufficient for the transformed equation to be proper.
	\end{remark}

	\comparison*
	\begin{proof}[Proof of Theorem \ref{thm:comparison}]
		Comparison is shown for the transformed DPE \eqref{eqn:transformed_dpe} with $\eta$ chosen as in Lemma \ref{lem:monotonicity}. Since the transformation \eqref{eqn:transformation} is sign-preserving, this is sufficient to establish comparison for \eqref{eqn:dpe} thanks to Lemma \ref{lem:transformed_dpe}. For the sake of simpler presentation later on, \eqref{eqn:transformed_dpe} is shortened to
		\[
			\begin{aligned}
			    \min \{ & f \tilde{V} - f^x \tilde{V}_x - f^\mu \tilde{V}_\mu - \Tr \Sigma D^2 \tilde{V}, \\
				    & \qquad \qquad \eta \tilde{V} + \tilde{V}_x - e^{-\eta x - \eta g(\mu)} \} = 0.
			\end{aligned}
		\]
		Note that the coefficients $f^x$, and $f^\mu$ are locally Lipschitz continuous on the interior of the domain, and the coefficient $f$ is continuous.

		Denote by $\tilde{u}$ and $\tilde{v}$ the functions transformed as in \eqref{eqn:transformation}. We note that these tend to 0 at infinity and that $\tilde{u}$ as well as $\tilde{v}$ are bounded.
		We distinguish between two cases:
		\begin{enumerate}
		    \item The function $\tilde{u} - \tilde{v}$ attains a maximum in $[0, \infty) \times \interior{\muspace}$. If the maximum is at $x=0$, we are done. Otherwise, the maximum is attained in the interior $\interior{\xmuspace}$.
		    \item \label{itm:bdryMax} There exists a maximizing sequence converging to a point $(\hat{x}, \hat{\mu})$ in $[0,\infty) \times \bdry \muspace$. Without loss of generality, assume that $\hat{\mu}$ is a lower boundary point. An upper boundary point is handled analogously. The method employed here originates in \cite{amadori2007}.

			For some small $\gamma > 0$, let $N = \{ (x, \mu) : \mu - \hat{\mu} < \gamma \} \cap ([0,\infty) \times \muspace)$ and define
			\[  \Psi_\delta(x,\mu) = \tilde{u}(x, \mu) - \tilde{v}(x, \mu) - \delta h(\mu), \quad (x,\mu) \in N, \]
			for $\delta \geq 0$ and
			\[  h(x, \mu) = \int_\mu^{\hat{\mu} + \gamma} e^{C (\xi - \hat{\mu})} (\xi - \hat{\mu})^{-1} \dif{\xi}, \quad (x,\mu) \in N, \]
			with
			\[  C = \sup \left\{ \frac{1}{\mu} - \frac{2 f^\mu(\mu)}{\musigma(\mu)^2} : 0 < \mu - \hat{\mu} < \gamma \right\}. \]
			Note that by Assumption \ref{ass:regularity}, $C < \infty$. It is easily verified that $h > 0$, $h(\mu) \to \infty$ as $\mu \to \hat{\mu}$, and that
			\[  f^\mu h' + \frac{\musigma^2}{2} h'' \leq 0, \quad \text{in } N. \]
			Hence, $\tilde{w}_\delta := \tilde{u} - \delta h$ is also a subsolution in $N$.

			Let $(x_n, \mu_n) \to (\hat{x}, \hat{\mu})$ be the maximizing sequence, and set $\delta = \delta_n = 1 / n h(\mu_n)$. Then $\delta \to 0$ as $\mu \to \hat{\mu}$. Thus,
			\[  \sup_N (\tilde{u} - \tilde{v}) \geq \sup_N \Psi_\delta \geq \Psi_\delta (x_n, \mu_n) = (\tilde{u} - \tilde{v})(x_n, \mu_n) - 1 / n, \]
			so
			\[  \lim_{\delta \to 0} \sup_N \Psi_\delta = \sup_N \left(\tilde{u} - \tilde{v}\right). \]
			Moreover, $\Psi_\delta$ attains a maximum $(x_\delta, \mu_\delta) \in \closure{N}$. For sufficiently small $\delta>0$, a maximum is attained in the interior, or otherwise a maximum of $\tilde{u} - \tilde{v}$ is attained for $\mu = \hat{\mu} + \gamma \in \interior{\xmuspace}$. It follows that
			\[  \sup_N (\tilde{u} - \tilde{v}) \geq (\tilde{u} - \tilde{v})(x_\delta, \mu_\delta) = \sup_N \Psi_\delta + \delta h(\mu_\delta) \geq \sup_N \Psi_\delta, \]
			so $\delta h(\mu_\delta) \to 0$. If $\max_N (\tilde{w}_\delta - \tilde{v}) \leq 0$,
			\[  \sup_\xmuspace (\tilde{u} - \tilde{v}) = \lim_{\delta \to 0} \max_N (\tilde{w}_\delta - \tilde{v}) \leq 0 \]
			follows. It therefore leads to no loss of generality to assume that $\tilde{u} - \tilde{v}$ attains a maximum in $\interior{\xmuspace}$.
		\end{enumerate}

		By the preceding argument, we may assume that $\tilde{u} - \tilde{v}$ attains local maximum $(x_0, \mu_0) \in \interior{\xmuspace}$. Let $N \subseteq N' \subseteq \xmuspace$ be two neighborhoods of $(x_0, \mu_0)$ in which this $(x_0, \mu_0)$ is a maximum and with $\closure{N} \subseteq N'$. For all $\epsilon > 0$, the function
		\[
		    \Phi_\epsilon(x, \mu, y, \nu) := \tilde{u}(x, \mu) - \tilde{v}(y, \nu) - \frac{1}{2 \epsilon} \left( \abs{x - y}^2 + \abs{\mu - \nu}^2 \right) - \norm{(x,\mu) - (x_0, \mu_0)}_2^4.
		\]
		has a maximum in $\closure{N} \times \closure{N}$, which we denote by $(x_\epsilon, \mu_\epsilon, y_\epsilon, \nu_\epsilon)$. In particular as $\epsilon \to 0$, the sequence converges along a subsequence.

		From the construction of $\Phi_\epsilon$ it follows that
		\[
		    \frac{1}{2 \epsilon} \left( \abs{x_\epsilon - y_\epsilon}^2 + \abs{\mu_\epsilon - \nu_\epsilon}^2 \right) \leq \tilde{u}(x_\epsilon, \mu_\epsilon) - \tilde{v}(y_\epsilon, \nu_\epsilon) - \max_{\closure{N}} (\tilde{u} - \tilde{v}).
		\]
		Observing that the superior limit of the right hand side is bounded from above by 0 yields
		\[
			\limsup_{\epsilon \rightarrow 0} \frac{1}{2 \epsilon} \left( \abs{x_\epsilon - y_\epsilon}^2 + \abs{\mu_\epsilon - \nu_\epsilon}^2 \right) \leq 0.
		\]
		This estimate is used later in the proof. Moreover, $(x_\epsilon, \mu_\epsilon) \to (x_0, \mu_0)$, which means they are local maxima in $N'$.

		By the Crandall--Ishii lemma, there exist matrices $X_\epsilon$ and $Y_\epsilon$ such that
		\[ \quad (p_\epsilon, X_\epsilon) \in \bar{J}^{2, +} u(x_\epsilon, \mu_\epsilon), \quad (p_\epsilon, Y_\epsilon) \in \bar{J}^{2, -} v(y_\epsilon, \nu_\epsilon) \]
		and
		\[  X_\epsilon \leq Y_\epsilon + \bigO\left(\frac{1}{\epsilon}\norm{(x,\mu) - (x_0, \mu_0)}_2^2 + \norm{(x,\mu) - (x_0, \mu_0)}_2^4\right), \]
		for
		\[ p_\epsilon = \frac{1}{\epsilon} (x_\epsilon - y_\epsilon, \mu_\epsilon - \nu_\epsilon).  \]
		In particular, $X_\epsilon \leq Y_\epsilon + \smallo(1)$ as $\epsilon \to 0$.
		Using the viscosity property of $\tilde{u}$ and $\tilde{v}$ yields
		\begin{multline*}
		    \min \Big\{ f(\mu_\epsilon) \tilde{u}(x_\epsilon, \mu_\epsilon)
			- f^x(\mu_\epsilon) \frac{x_\epsilon - y_\epsilon}{\epsilon} - f^\mu(\mu_\epsilon) \frac{\mu_\epsilon - \nu_\epsilon}{\epsilon} - \Tr \Sigma(\mu_\epsilon) D^2 \tilde{X_\epsilon}, \\
			\eta \tilde{u}(x_\epsilon, \mu_\epsilon) + \frac{x_\epsilon - y_\epsilon}{\epsilon} - e^{-\eta x_\epsilon - \eta g(\mu_\epsilon)} \Big\} \leq 0.
		\end{multline*}
		and
		\begin{equation}
			\label{eqn:supersolution}
			\begin{multlined}[c][0.9\textwidth]
			    \min \Big\{ f(\nu_\epsilon) \tilde{v}(y_\epsilon, \nu_\epsilon)
				- f^x(\nu_\epsilon) \frac{x_\epsilon - y_\epsilon}{\epsilon} - f^\mu(\nu_\epsilon) \frac{\mu_\epsilon - \nu_\epsilon}{\epsilon} - \Tr \Sigma(\nu_\epsilon) D^2 \tilde{Y_\epsilon}, \\
			    \eta \tilde{v}(y_\epsilon, \nu_\epsilon) + \frac{x_\epsilon - y_\epsilon}{\epsilon} - e^{-\eta y_\epsilon - \eta g(\nu_\epsilon)} \Big\} \geq 0.
			\end{multlined}
		\end{equation}
		The rest of the proof is dividend into two cases, depending on whether
		\begin{equation}
			\label{eqn:case1}
			\begin{multlined}[c][0.9\textwidth]
			    f(\mu_\epsilon) \tilde{u}(x_\epsilon, \mu_\epsilon)
			    - f^x(\mu_\epsilon) \frac{x_\epsilon - y_\epsilon}{\epsilon} - f^\mu(\mu_\epsilon) \frac{\mu_\epsilon - \nu_\epsilon}{\epsilon} - \Tr \Sigma(\mu_\epsilon) D^2 \tilde{X_\epsilon} \leq 0
			\end{multlined}
		\end{equation}
		or
		\begin{equation}
			\label{eqn:case2}
			\eta \tilde{u}(x_\epsilon, \mu_\epsilon) + \frac{x_\epsilon - y_\epsilon}{\epsilon} - e^{- \eta x_\epsilon - \eta g(\mu_\epsilon)} \leq 0.
		\end{equation}

		{\it Case 1.} Assume \eqref{eqn:case1}. Using \eqref{eqn:supersolution} we arrive at
		\begin{multline*}
		    f(\mu_\epsilon) (\tilde{u}(x_\epsilon, \mu_\epsilon) - \tilde{v}(y_\epsilon, \nu_\epsilon))
		    + (f(\mu_\epsilon) - f(\nu_\epsilon)) \tilde{v}(y_\epsilon, \nu_\epsilon) \\
		    - (f^x(\mu_\epsilon) - f^x(\nu_\epsilon)) \frac{x_\epsilon - y_\epsilon}{\epsilon}
		    - (f^\mu(\mu_\epsilon) - f^\mu(\nu_\epsilon)) \frac{\mu_\epsilon - \nu_\epsilon}{\epsilon} \\
		    -\Tr (\Sigma(\mu_\epsilon) - \Sigma(\nu_\epsilon))Y_\epsilon \leq \Tr \Sigma(\mu_\epsilon) (X_\epsilon - Y_\epsilon) \in \smallo(1)
		\end{multline*}
		Using the (local) Lipschitz continuity of the coefficients as well as the quadratic convergence rates of $x_\epsilon - y_\epsilon$ and $\mu_\epsilon - \nu_\epsilon$, we find that
		\begin{equation*}
			\begin{multlined}[c][0.9\textwidth]
			    f(\mu_0) (\tilde{u} - \tilde{v})(x_0, \mu_0)
			    = \limsup_{\epsilon \rightarrow 0} f(\mu_\epsilon) (\tilde{u}(x_\epsilon, \mu_\epsilon) - \tilde{v}(y_\epsilon, \nu_\epsilon)) \leq 0
			\end{multlined}
		\end{equation*}
		We conclude that
		\[
			0 \geq (\tilde{u} - \tilde{v})(x_0, \mu_0).
		\]

		{\it Case 2.} Assume \eqref{eqn:case2}. Using \eqref{eqn:supersolution} yields
		\[
		    \eta (\tilde{u}(x_\epsilon, \mu_\epsilon) - \tilde{v}(y_\epsilon, \nu_\epsilon)) \leq e^{- \eta x_\epsilon - \eta g(\mu_\epsilon)} - e^{-\eta y_\epsilon - \eta g(\nu_\epsilon)}.
		\]
		Once again we use the convergence results, and once again we conclude that
		\[
			0 \geq (\tilde{u} - \tilde{v})(x_0, \mu_0) = \max_\xmuspace (\tilde{u} - \tilde{v}).
		\]

		The inequality $\tilde{u} \leq \tilde{v}$ holds at any local maximum. Moreover, as mentioned in the beginning of the proof, we may assume that $\tilde{u} - \tilde{v}$ is attains its maximum on the interior.
		The theorem statement follows from the fact that
		\[
			\tilde{u} \leq \tilde{v} \Longleftrightarrow u \leq v.
		\]
	\end{proof}

\bibliographystyle{plainnat}
\bibliography{optimaldividends}

\end{document}